\documentclass[a4paper]{scrartcl}

\usepackage[utf8]{inputenc}
\usepackage[T1]{fontenc}
\usepackage[english]{babel}
\usepackage{amsmath,amssymb,amsfonts,siunitx,commath}
\usepackage{amsthm}
\usepackage{tikz,url}
\usepackage{geometry}
\usepackage{mathtools}
\mathtoolsset{showonlyrefs}
\usepackage{subcaption}
\usepackage{algorithm}
\usepackage{algpseudocode}
\usepackage{hyperref}
\usepackage{enumerate}
\usepackage{listings}
\usepackage{esint}

\DeclareMathOperator*{\argmin}{arg\,min}

\newcommand{\R}{\mathbb{R}}

\newcommand{\Z}{\mathbb{Z}}

\newcommand{\Haus}{\mathcal{H}}

\renewcommand\d{\mathrm{d}}

\newcommand\tT{\mathrm{T}}

\newcommand{\ones}{\mathbf{1}}
\newcommand{\pp}{{\prec\!\prec}}
\newcommand{\Gr}{\textup{Gr}\,}
\newcommand{\mres}{\mathbin{\vrule height 1.6ex depth 0pt width 0.13ex\vrule height 0.13ex depth 0pt width 1.3ex}}

\newcommand{\wto}{\rightharpoonup}
\newcommand{\wtos}{\stackrel{*}{\rightharpoonup}}
\DeclareMathOperator{\Div}{div}
\newcommand{\Bx}{\mathbf{x}}
\newcommand{\By}{\mathbf{y}}

\theoremstyle{plain}
\newtheorem{lemma}{Lemma}
\newtheorem{theorem}[lemma]{Theorem}
\newtheorem{corollary}[lemma]{Corollary}
\newtheorem{proposition}[lemma]{Proposition}
\newtheorem{remark}[lemma]{Remark}
\newtheorem{example}[lemma]{Example}

\newtheorem{definition}[lemma]{Definition}
\begin{document}
\title{Limit Points of Reflow with Minibatch Optimal Transport}
\author{Antonin Chambolle\footnote{Universit\'e Paris Dauphine - PSL \& Inria, \texttt{chambolle@ceremade.dauphine.fr}}\and Johannes Hertrich\footnote{ENS Paris, \texttt{johannes.hertrich@ens.fr}}}
\maketitle

\begin{abstract}
Rectified flows, also called flow matching or stochastic interpolants, are generative models that learn a time-dependent vector field steering a probability curve between two probability distributions, usually referred to as latent and target distributions. Reflow accelerates inference by iteratively straightening the trajectories induced by this vector field. We study the asymptotic behavior of this iteration and characterize its limit points. First, we define weak rectified couplings which always exist. Next, when rectified flow updates are alternated with minibatch optimal transport steps of fixed batch size, we show that any limit is $N$-cyclically monotone, where $N$ is the batch size. 
Such $N$-cyclically monotone couplings enjoy favorable structural and stability properties such as rectifiability and straightness.
Finally, restricting velocities to gradient fields and assuming additional support conditions, we prove that reflow limits coincide with the optimal transport map between the endpoint distributions. 
\end{abstract}

\section{Introduction}

Generative modeling is the task of learning a process to sample from a target probability measure $\mu_1$ on $\R^d$, which is only given by data. A typical approach for doing so is to consider an easy-to-sample latent distribution $\mu_0$ and to learn a transport from $\mu_0$ to $\mu_1$. In practice, $\mu_0$ is often chosen as a standard Gaussian, but depending on the application other distributions can be used.
Examples for generative models include variational autoencoders \cite{KW2013}, normalizing flows \cite{rezende15} and score-based diffusion methods \cite{ho2020denoising,song2020score}.
\smallskip

In this paper, we consider rectified flows \cite{L2022,LCL2023}, which were introduced at the same time under the names flow matching \cite{LCBNL2023} and stochastic interpolants \cite{AV2021}. We refer also to \cite{ABV2023, lipman2024flow, pierret2026flow, wald2025flow} for some introduction and overview papers. 
Rectified flows start with some coupling $\gamma\in\Gamma(\mu_0,\mu_1)\coloneqq \{\gamma\in \mathcal P_2(\R^d\times \R^d):{\pi_x}_\#\gamma=\mu_0, {\pi_y}_\#\gamma=\mu_1\}$, where $\pi_x(x,y)=x$ and $\pi_y(x,y)=y$, and learn a velocity field by minimizing the loss function
\begin{equation}\label{eq:flow_matching_loss}
v_t\in\argmin_{w_t\in L^2(\mu_t)}\mathcal L(w_t|\gamma)\coloneqq\int_0^1\int_{\R^d\times\R^d}\|w_t((1-t) x+ t y)-y+x\|^2\d\gamma(x,y)\d t,
\end{equation}
where $\mu_t$ is the probability measure defined by 
\begin{equation}\label{eq:def_mu_t}
\int_{\R^d}\psi(x)\d\mu_t(x)=\int_{\R^d\times\R^d} \psi((1-t)x+ty)\d\gamma(x,y)\quad\text{for all}\quad \psi\in C_c^\infty(\R^d).
\end{equation}
Then, it can be shown that $\mu_t$ and the velocity field $v_t$ from \eqref{eq:flow_matching_loss} fulfill the continuity equation
\begin{equation}\label{eq:CE}
    \partial_t \mu_t + \Div(v_t\mu_t)=\delta_0\otimes \mu_0-\delta_1\otimes \mu_1
\end{equation}
in a weak sense, i.e., it holds for all $\psi\in C_c^\infty(\R\times\R^d)$ that
\begin{align}\label{eq:weak_CE}
\int_0^1\int_{\R^d}(\partial_t\psi(t,x))\d\mu_t(x)\d t+\int_0^1\int_{\R^d}\langle\nabla \psi(t,x),v_t(x)\rangle\d\mu_t(x)\d t\\
=\int_{\R^d}\psi(1,x)\d\mu_1(x)-\int_{\R^d}\psi(0,x)\d\mu_0(x).
\end{align}
If $v_t$ is regular enough, 
this allows sampling from $\mu_1$ by sampling $x_0$ from $\mu_0$ and computing the solution $\Phi_1(x_0)$ of the flow ODE
\begin{equation}\label{eq:flow_ode}
\partial_t \Phi_t(x)=v_t(\Phi_t(x)),\quad \Phi_0(x)=x.
\end{equation}
The authors of \cite{LCL2023} call a coupling \emph{rectifiable} if the solution of \eqref{eq:flow_ode} always exists and is unique. However, rectifiability is not directly clear from the coupling and can be violated even in simple examples, see \cite{HCD2025}.

In order to obtain straight velocity fields $v_t$, Liu et al.~\cite{L2022,LCL2023} proposed to iterate the rectified flow coupling, which is also known under the name ``reflow''.
The authors of \cite{RBSR2024} derive a sufficient condition under which reflow leads to straight couplings in the second step, but point out that this condition does not hold true in general. A similar result in a more general setting is derived in \cite{PSM2026}.

Reflow can be related to optimal transport by choosing velocity fields which admit a potential. Under certain conditions, \cite{L2022, WXLZ2026} showed that reflow leads to the optimal transport plan. However, this result requires strong regularity assumptions, which are impossible to check during training time and are not fulfilled in many common cases, see also \cite{HCD2025} for more details.
Another approach to relate rectified flows to optimal transport is minibatch OT as proposed by Pooladian et al.~\cite{PBDALC2023} and Tong et al.~\cite{TFMH2024}. This describes the techniques to reorder batches from the coupling $\gamma$ according to discrete optimal transport. While it is clear that minibatch OT does not directly lead to optimal transport, several papers studied the convergence if the batch size tends to infinity, see \cite{BDN2026,fukumizu2026flow,PBDALC2023}. While \cite{PBDALC2023} and \cite{fukumizu2026flow} focus on discrete settings, \cite{BDN2026} proves a general convergence result to optimal transport as the batch size tends to infinity.

\paragraph{Outline and Contributions.} In this paper we study the existence and convergence of rectified flows and the reflow procedure beyond the common regularity assumptions. We pay particular attention to the case where minibatch OT is applied in every iteration and show that this provides more tractable structures. The outline will be as follows:
\smallskip

In Section~\ref{sec:smirnov} we present a weak rectified coupling which always exists even when the coupling is not rectifiable in the sense that the flow ODE \eqref{eq:flow_ode} does not have a unique solution.
In the case that $\gamma$ is rectifiable, the weak rectified coupling is unique and coincides with the strong rectified coupling defined by the flow ODE.
In the other cases, the weak rectified coupling might not be unique, but we show that any weak solution admits properties which make it as useful as a strong rectified coupling.
In the context of the reflow procedure, we give an example that the loss function $\gamma \mapsto \inf_{w_t\in L^2(\mu_t)}\mathcal L(w_t|\gamma)$ is not lower semicontinuous in $\gamma$ (wrt weak convergence) such that it is not clear whether limit points are straight.
\smallskip

In Section~\ref{sec:minibatch} we draw our attention to minibatch OT and its application in the reflow procedure. To this end, we first formally define minibatch OT as an operator on the space of couplings and study its fixed points. In the context of reflow, we characterize limit points, where minibatch OT with fixed batch size is applied in every step. We find that these limit points are $N$-cyclically monotone, where $N$ is the chosen batch size and that this implies that the limit coupling is straight and rectifiable.
\smallskip

In Section~\ref{sec:gradient}, we consider reflow, where we apply both minibatch OT and the constraint from \cite{L2022} that velocities admit a potential. We consider the limit and fixed points of reflow and prove that they coincide with optimal transport under significantly weaker regularity assumptions than in \cite{L2022}. Under strong assumptions, we also prove convergence of the iterates to optimal transport.
\smallskip

Finally, we illustrate the influence of the batch size on the limit points numerically in Section~\ref{sec:numerics}. Further, conclusions, limitations and future work are discussed in Section~\ref{sec:conclusions}. Additional examples and some proofs are given in the appendix.

\section{Weak Rectified Couplings and Straight Limits}\label{sec:smirnov}

In \cite{L2022,LCL2023}, the authors define rectified couplings as follows.

\begin{definition}\label{def:strong_rec}
We call a coupling $\gamma$ with velocity field $v_t\in\argmin_{w_t\in L^2(\mu_t)} \mathcal L(w_t|\gamma)$ rectifiable if the flow ODE \eqref{eq:flow_ode}
has a unique solution for $\mu_0$-almost every $x$.
For a rectifiable coupling $\gamma$, we call the coupling $\tilde \gamma$ defined by
\begin{equation}\label{eq:strong_rectification}
\int \psi(x,y)\d\tilde \gamma(x,y)=\int \psi(x,\Phi_1(x)) \d \mu_0(x)\quad\text{for all}\quad \psi\in C_c^\infty(\R^d\times \R^d),
\end{equation}
the strong rectified coupling of $\gamma$.
\end{definition}
In the case that $\mu_0$ is Gaussian, \cite{RBSR2024} formulate sufficient conditions for the independent coupling to be rectifiable. A consideration for more general marginals $\mu_0$ is done in \cite{mena2025statistical}. However, without further regularity assumptions, not even the independent coupling is rectifiable in general.
In this section,
we consider a weak formulation of rectified couplings, which always exists even when the flow ODE does not admit
a unique solution. 
The weak rectified coupling relies on a representation of solutions of the continuity equation as a probability measure on the space of rectifiable paths. This representation is described in \cite[Section 8.2]{AGS2008} and can also be derived from Smirnov's decomposition theorem \cite{Smirnov93}.
We show how this decomposition can be used for defining weak rectified couplings in Subsection~\ref{subsec:weak_rec}. Afterwards, we show in Subsection~\ref{subsec:weak_reflow} that any limit point of reflow (with weak rectified coupling) is straight.

\subsection{Weak Rectified Couplings}\label{subsec:weak_rec}

We want to define a weaker version of Definition~\ref{def:strong_rec}, which always exists. To this end, we start with a solution $(\mu_t,v_t)$ of the continuity equation $\partial_t \mu_t + \Div(v_t\mu_t)=0$ (in the weak sense \eqref{eq:weak_CE}) with $\int_0^1\|v_t\|^2_{L^2(\mu_t)}\d t<\infty$. In our application, $\mu_t$ will be the interpolation \eqref{eq:def_mu_t} and $v_t$ the minimizer of the flow matching loss \eqref{eq:flow_matching_loss}.

Even when the flow ODE \eqref{eq:flow_ode} does not admit a unique solution, it is proven in~\cite[Thm.~8.2.1]{AGS2008} (see also
eq.~(8.2.8)) that there is a probability measure $\Lambda$, defined
on integral curves $\lambda:[0,1]\to\R^d$ such that it holds
\begin{align}
    \dot{\lambda}(t) &= v_t(\lambda(t))\quad\text{for a.e.~$t$ and $\Lambda$-a.e.~$\lambda$, and,}\label{eq:Lambda1}\\
    \int_{\R^d} \psi(x)\d\mu_t &= \int \psi(\lambda(t)) \d\Lambda(\lambda)\quad\text{for any $\psi\in C_c(\R^d)$}.\label{eq:Lambda2}
\end{align}
If the flow ODE \eqref{eq:flow_ode} has a unique solution $\Phi_t$, these paths $\lambda$ coincide with the trajectories of the flow ODE in the sense that $\Lambda=f_\#\mu_0$, where $f$ maps a point $x\in\R^d$ to the curve $t\mapsto\Phi_t(x)$.
As observed for instance in~\cite{StepanovTrevisan}, this decomposition
can also be deduced as a
consequence of Smirnov's theorem~\cite[Thm.~C]{Smirnov93}.

With this measure $\Lambda$ at hand, we can, in analogy to Definition~\ref{def:strong_rec}, introduce a weak rectified coupling that is determined by the initial and terminal points of the lines $\lambda$ in the support of $\Lambda$.
We emphasize that the measure $\Lambda$ from \cite[Thm.~8.2.1]{AGS2008} is not unique, so that weak rectified couplings are not unique either.

\begin{definition}[Weak Rectified Couplings]\label{def:weak_rec} 
Let $\gamma\in\Gamma(\mu_0,\mu_1)$ and $\mu_t$ be defined by \eqref{eq:def_mu_t} and denote by $v_t=\argmin_{w_t\in L^2(\mu_t)}\mathcal L(w_t|\gamma)$ the solutions of the flow matching loss \eqref{eq:flow_matching_loss}.
Then, we define the set of weak rectified couplings $\mathcal R(\gamma)$ as the set of all transport plans $\tilde \gamma$ given by
\begin{equation}\label{eq:gamma_tilde}
\int_{\R^d\times \R^d} \psi(x,y)  \d\tilde\gamma=  \int \psi(\lambda(0),\lambda(1)) \d\Lambda(\lambda),\quad\text{for all }\psi\in C_c^\infty(\R^d\times\R^d),
\end{equation}
where $\Lambda$ fulfills \eqref{eq:Lambda1} and \eqref{eq:Lambda2} with respect to $v_t$.
\end{definition}

We provide an explicit example for the non-uniqueness of the weak rectified coupling in Appendix~\ref{app:non_uniqueness_rect}.
Next, we collect some properties of weak rectified couplings in the following corollary. The proof follows directly from the definition and from \cite[Thm 8.2.1]{AGS2008}.

\begin{corollary}[Properties of Weak Rectified Couplings]\label{cor:properties_weak_rec}
Let $\gamma\in\Gamma(\mu_0,\mu_1)$ and let $\tilde \gamma\in\mathcal R(\gamma)$. Then the following holds true.
\begin{enumerate}
    \item[(i)] $\mathcal R(\gamma)$ is a non-empty subset of $\Gamma(\mu_0,\mu_1)$.
    \item[(ii)] If the flow ODE \eqref{eq:flow_ode} has a unique solution, then $\mathcal R(\gamma)$ is single-valued and its element coincides with the strong rectified coupling from Definition~\ref{def:strong_rec}.
    \item[(iii)] $\int \|x-y\|^2\d \tilde \gamma(x,y)\leq \int_0^1\int \|v_t(x)\|^2\d\mu_t(x)\d t\leq \int\|x-y\|^2\d\gamma(x,y)$.
    \item[(iv)] It holds $\int \|x-y\|^2\d \tilde \gamma(x,y)= \int_0^1\int \|v_t(x)\|^2\d\mu_t(x)\d t$ if and only if $\Lambda$-a.e. $\lambda$ is a straight line (where $\Lambda$ generates $\tilde \gamma$ by \eqref{eq:gamma_tilde}).
\end{enumerate}
\end{corollary}
\begin{proof}
Part (i) follows directly from \cite[Thm 8.2.1]{AGS2008} and \eqref{eq:Lambda2}. Further, by \eqref{eq:Lambda1}, $\Lambda$-a.e. curve $\lambda$ solves (a.e.) the flow ODE \eqref{eq:flow_ode}. Hence, if the flow ODE has a unique solution, then any weak rectified coupling coincides with the strong rectified coupling. For part (iii) and (iv) one derives from \eqref{eq:Lambda1} and \eqref{eq:Lambda2} that
\begin{multline*}
\int_0^1\int_{\R^d} \|v_t(x)\|^2\d\mu_t\d t = 
\int_0^1 \int \|v_t(\lambda(t))\|^2 \d\Lambda(\lambda) \d t
\\= \int \int_0^1 \|\dot{\lambda}(t)\|^2 \d t\d\Lambda(\lambda)
\ge \int \|\lambda(1)-\lambda(0)\|^2 \d\Lambda(\lambda).
\end{multline*}
In addition, equality holds in the last inequality if and only if
for $\Lambda$-a.e.~curve $\lambda$, $\dot{\lambda}(t) = \lambda(1)-\lambda(0)$
for a.e.~$t\in [0,1]$ (so that $\lambda(t) = (1-t)\lambda(0)+t\lambda(1)$).  
\end{proof}
We note that similar to the strong rectified coupling considered in \cite{L2022,LCL2023}, we get by the same proof that part (iii) from the corollary holds for any convex cost function $c\colon\R^d\to\R\cup\{\infty\}$ instead of $\|\cdot\|^2$. The equality condition (iv) holds true for any strictly convex cost function $c\colon\R^d\to\R\cup\{\infty\}$ instead of $\|\cdot\|^2$.

\subsection{Reflow and Straight Couplings}\label{subsec:weak_reflow}

In order to reduce the computational cost of evaluating rectified flows, we are interested in  couplings for which the trajectories of the flow ODE correspond to straight lines. If a coupling is straight and rectifiable, the flow ODE \eqref{eq:flow_ode} of such straight couplings can be represented as $\Phi_t(x)=x+tv_0(x)$, which requires only a single evaluation of the velocity field and is therefore computationally efficient. Conceptually, the computation of straight couplings is closely related to distillation methods \cite{luhman2021knowledge, meng2023distillation, salimans2022progressive} of diffusion or flow models.
Following \cite{L2022,LCL2023}, we use the following formal definition for straight couplings.
\begin{definition}
We call a coupling $\gamma\in\Gamma(\mu_0,\mu_1)$ a straight coupling if the velocity $v_t\in\argmin_{w_t\in L^2(\mu_{t})} \mathcal L(w_t|\gamma)$ fulfills $\mathcal L(v_t|\gamma)=0$, or equivalently $v_t((1-t)x+ty)=y-x$ for $\d t\otimes\gamma$-almost every $(t,x,y)$.
\end{definition}
If $\gamma$ is additionally rectifiable, then this definition is equivalent to assuming that the flow ODE can be represented as $\Phi_t(x)=x+tv_0(x)$.
\smallskip

In order to find straight couplings, \cite{L2022,LCL2023} propose the reflow algorithm, which starts at an arbitrary initial coupling $\gamma_0\in\Gamma(\mu_0,\mu_1)$ and then iteratively computes a sequence of couplings $\gamma_{n+1}\in\mathcal R(\gamma_n)$. They show that the loss function $\min_{v_t\in L^2(\mu_{t,n})} \mathcal L(v_t|\gamma_n)$ converges to zero. However, we point out two problems, which show that this does not necessarily lead to a one-step generation model in the limit.
\smallskip

First, limit points of reflow are not necessarily rectifiable. An example is given by the coupling defined by $\int\psi(x,y)\d\gamma(x,y)=\int\psi(x,-x)\d\mathcal N(0,I)(x)$ which was considered in \cite[Section 4.2]{HCD2025} as an example for non-rectifiable couplings, see also Appendix~\ref{app:non_uniqueness_rect}. This coupling fulfills $\gamma\in\mathcal R(\gamma)$ and is consequently also a limit point of reflow.\smallskip

Second, the loss function $\min_{v_t\in L^2(\mu_{t})} \mathcal L(v_t|\gamma)$ is not lower semicontinuous. In particular, we cannot conclude that limit points of the sequence $(\gamma_n)_n$ in the reflow algorithm are really straight couplings. Indeed, we include in Appendix~\ref{app:not_lsc} an explicit example showing that not even limits of straight couplings are necessarily straight. However, the sequence in this example is not generated by the reflow algorithm, and it remains open whether a similar example can be constructed along reflow sequences.

\begin{remark}
Straight line flows were also studied in the papers \cite{RBSR2024,PSM2026}. In particular, these papers hint that in some cases the iterates $\gamma_n$ might arrive at the limit already after two (or finitely many) steps. Specifically, the authors of \cite{RBSR2024} study more closely the case of the independent initialization $\gamma_0=\mu_0\otimes\mu_1$ with $\mu_0=\mathcal N(0,I)$. Then, they conjecture that already $\gamma_1\in\mathcal R(\gamma_0)$ is a straight coupling in the sense that $\min_{v_t}\mathcal L(v_t|\gamma_1)=0$. However, they only prove this result for the case where $\gamma_1$ is supported on the graph of a function $f$ with $\nabla f(x)+\nabla f(x)^\tT\geq 0$ for all $x$ 
(that is, a $2$-monotone graph in the sense introduced below) and experimentally report that this assumption is sometimes violated in practice. Indeed, \cite{PHHD2026} verifies that this assumption is in general not true.
A similar monotonicity assumption on straightness is considered in \cite{PSM2026}, where also more general interpolation paths than just $X_t=(1-t)X_0+tX_1$ are considered.
\end{remark}

\section{Reflow with Minibatch OT}\label{sec:minibatch}

In order to achieve couplings closer to optimal transport, several papers proposed to combine rectified flows
with minibatch OT \cite{PBDALC2023,TFMH2024}, which consists
in a minibatch based stochastic gradient descent for~\eqref{eq:flow_matching_loss} where the batches $(x^{(i)},y^{(i)})_{i=1}^N$ are, in a preliminary step, reordered in order to minimize their Wasserstein cost.
In this section, we view minibatch OT as an operator in the space of couplings. We study its fixed points and establish convergence results of reflow if minibatch OT is applied in each step.

\subsection{Minibatch OT as Fixed Point Iteration}

Let $\gamma\in\Gamma(\mu_0,\mu_1)$ be a coupling and let $N\in\Z_{>1}$.
Moreover, we denote by $S=(S_1,...,S_N)\colon(\R^d\times\R^d)^N\to(\R^d\times\R^d)^N$ a mapping computing a discrete
optimal coupling with respect to the quadratic cost. That is, $S$ maps an input $(\Bx,\By)\coloneqq (x^{(i)},y^{(i)})_{i=1}^N$ to $(x^{(i)},y^{(\sigma(i))})_{i=1}^N$ such that 
\begin{equation}\label{eq:opt_permutation}
\sigma\in\argmin_{\tau\in\Sigma_N}\sum_{i=1}^N \|x^{(i)}-y^{(\tau(i))}\|^2
\end{equation}
where $\Sigma_N$ is the group of all permutations of length $N$. 
Since the optimal $\sigma$ is not always unique, there exist several mappings $S$ fulfilling \eqref{eq:opt_permutation}. We choose an arbitrary (but fixed) one among them; the most natural is maybe to consider the one which ``moves the points less'', that is, which minimizes $\sum_i \|y^{(i)}-y^{(\tau(i))}\|^2$ among all permutations $\tau$ minimizing~\eqref{eq:opt_permutation}.
Now, we define the minibatch OT as the coupling given by
\begin{equation}
\mathcal F_N(\gamma)\coloneqq\frac1N\sum_{i=1}^N {S_i}_\#\gamma^{\otimes N},
\label{eq:shuffleplan}
\end{equation}
where $\gamma^{\otimes N}=\gamma\otimes\cdots\otimes\gamma$.

We will now investigate the fixed point set of $\mathcal F_N$. To this end, we consider the following
definition, which is a weaker version of $c$-cyclical monotonicity in optimal transport, cf.~\cite[Def 5.1]{Villani2008}.

\begin{definition}
We call a set $A\subset\R^d\times\R^d$ $N$-cyclically monotone for $N\geq 2$ if it holds for all $(x^{(i)},y^{(i)})\in A$ and
$\tau\in\Sigma_N$ that
$$
\sum_{i=1}^N \|x^{(i)}-y^{(i)}\|^2\leq \sum_{i=1}^N \|x^{(i)}-y^{(\tau(i))}\|^2.
$$
Similarly, we call a transport plan $\gamma\in \mathcal P_2(\R^d\times\R^d)$ $N$-cyclically monotone
if $\mathrm{supp}(\gamma)$ is $N$-cyclically monotone.
\end{definition}

An equivalent condition is that $A$ is $N$-cyclically monotone if and only if it holds for all $(x^{(i)},y^{(i)})\in A$ that
$
\sum_{i=1}^N\langle x^{(i)},y^{(i)}-y^{(i+1)}\rangle\geq0.
$
By definition, $N$-cyclical monotonicity implies $M$-cyclical monotonicity for all $2\leq M<N$.
Further, by \cite[Thm 5.9, Rem 5.10]{Villani2008} a coupling is optimal if and only if it is $N$-cyclically monotone \emph{for all $N\in\Z_{>1}$}.
However, there are simple counterexamples that $N$-cyclical monotonicity \emph{for one specific $N$} does not imply that
the corresponding transport map is optimal. In two dimensions, such examples can be constructed by certain rotations as shown in the following example from \cite{CD2014,CG2003}.
\begin{example}
    Let $\mu_0=\mu_1=\mathcal N(0,\mathrm{Id})$ in $d=2$ dimensions. Then, we define 
    $$
    \gamma_N=(\mathrm{Id},u)_\#\mu_0,\quad\text{for}\quad u(x)=Ax\quad\text{with}\quad A=\left(\begin{array}{cc}\cos\big(\frac{\pi}{N}\big)&-\sin\big(\frac{\pi}{N}\big)\\
    \sin\big(\frac{\pi}{N}\big)&\cos\big(\frac{\pi}{N}\big)
    \end{array}\right).
    $$
    Then, it was shown in \cite{CG2003} that $\gamma_N$ is $N$-cyclically monotone but not $N+1$-cyclically monotone.
\end{example}

The next proposition derives some basic properties of the mapping $\mathcal F_N$ representing the minibatch OT.

\begin{proposition}\label{prop:minibatchOT}
Let $\gamma\in \Gamma(\mu_0,\mu_1)$ and let $\gamma^{(N)}=\mathcal F_N(\gamma)$.
Then, the minibatch OT has the following properties.
\begin{itemize}
\item[(i)] \textbf{Marginal Preservation:} It holds that $\gamma^{(N)}\in\Gamma(\mu_0,\mu_1)$.
\item[(ii)] \textbf{Reduced Transport Cost:} It holds $\int\|x-y\|^2\d\gamma^{(N)}(x,y)\leq \int\|x-y\|^2\d\gamma(x,y)$.
\item[(iii)] \textbf{Equality:} We have $\int\|x-y\|^2\d\gamma^{(N)}(x,y)= \int\|x-y\|^2\d\gamma(x,y)$ if and only if $\gamma$ is $N$-cyclically monotone.
\item[(iv)] \textbf{Continuous Cost:} The mapping $(\mathcal P_2(\R^d\times\R^d),W_2)\to(\R,|\cdot|)$ defined by $\gamma\mapsto \int\|x-y\|^2\d\mathcal F_N(\gamma)(x,y)$ is continuous.
\end{itemize}
\end{proposition}
\begin{proof}
\begin{itemize}
\item[(i)] 
Let $\pi_x(x,y)=x$ and $\pi_y(x,y)=y$. Then, we have by definition that $\pi_x\circ S_i((x^{(i)},y^{(i)})_{i=1}^N)=x^{(i)}$, so that $(\pi_x\circ S_i)_\#\gamma^{\otimes N}={\pi_x}_\#\gamma=\mu_0$. In particular, we have that 
$$
{\pi_x}_\#\mathcal F_N(\gamma)=\frac1N\sum_{i=1}^N(\pi_x\circ S_i)_\#\gamma^{\otimes N}=\frac1N\sum_{i=1}^N\mu_0=\mu_0.
$$
For the other marginal $\mu_1^{(N)}$ of $\gamma^{(N)}$, note that for any $\psi\in C_c^\infty(\R^d)$ it holds
\begin{align}
\int \psi(y)\d\mu_1^{(N)}(y)&=\int \psi(y)\d\gamma^{(N)}(x,y)=\frac1N\sum_{i=1}^N\int \psi(y)\d {S_i}_\#\gamma^{\otimes N}(x,y)\\
&= \frac1N\sum_{i=1}^N\int \psi\left(\pi_y\circ S_i(\Bx,\By)\right)\d \gamma^{\otimes N}(\Bx,\By).
\end{align}
Since by definition $\pi_y\circ S_i(\Bx,\By)$ is a permutation of $y^{(i)}$ this is equal to
$$
 \frac1N\sum_{i=1}^N\int \psi(y^{(i)})\d \gamma^{\otimes N}(\Bx,\By)=\int \psi(y)\d\gamma(x,y)=\int \psi(y)\d\mu_1(y).
$$
Together, the two previous formulas imply that $\mu_1^{(N)}=\mu_1$.
\item[(ii)] Denote $\gamma^{(N)}=\mathcal F_N(\gamma)$. Then, it holds that
\begin{align}
\int \|x-y\|^2\d \gamma(x,y)&=\frac1N \int \sum_{i=1}^N \|x^{(i)}-y^{(i)}\|^2 \d \gamma^{\otimes N}(\Bx,\By)\\
&\geq\frac1N \int\min_{\tau\in\Sigma_N}\sum_{i=1}^N \|x^{(i)}-y^{(\tau(i))}\|^2 \d \gamma^{\otimes N}(\Bx,\By)\\
&=\frac1N \sum_{i=1}^N\int \|x-y\|^2 \d {S_i}_\#\gamma^{\otimes N}(x,y)\\
&=\int \|x-y\|^2 \d \gamma^{(N)}(x,y).
\end{align}
\item[(iii)] In (ii) equality holds if and only if it holds
\[
\sum_{i=1}^N \|x^{(i)}-y^{(i)}\|^2\leq \sum_{i=1}^N \|x^{(i)}-y^{(\tau(i))}\|^2
\]
for every $\tau\in\Sigma_N$ for $\gamma^{\otimes N}$-almost every $(\Bx,\By)$.
\item[(iv)] Note that the mapping $(\Bx,\By)\mapsto \min_{\tau\in\Sigma_N}\frac1N\sum_{i=1}^N\|x^{(i)}-y^{(\tau(i))}\|^2$
is continuous (as the minimum over finitely many continuous functions). Thus the functional
\begin{equation}\label{eq:funny_cost}
\eta\mapsto \frac1N\int\min_{\tau\in\Sigma_N}\sum_{i=1}^N\|x^{(i)}-y^{(\tau(i))}\|^2\d\eta(\Bx,\By)
\end{equation}
for $\eta\in\mathcal P_2((\R^d\times\R^d)^N)$
is the objective value of an optimal transport problem with $2N$ marginals and continuous cost function.
Following \cite[Lem 4.3]{Villani2008} \eqref{eq:funny_cost} is lsc and similarly to \cite[Cor 6.9]{Villani2008} we derive that \eqref{eq:funny_cost} is continuous.
Now, $\gamma\mapsto \int\|x-y\|^2\d\mathcal F_N(\gamma)(x,y)$ is the composition of the continuous mappings \eqref{eq:funny_cost} and $\gamma\mapsto\gamma^{\otimes N}$
and therefore continuous.
\end{itemize}
\end{proof}

The coupling $\mathcal F_N(\gamma)$ might depend on the choice of the discrete OT solution map $S$ on the set, where the discrete OT has more than one solution. Indeed, the following example shows that $\mathcal F_N(\gamma)$ can lead to different couplings for different choices of $S$.

\begin{example}
We consider the coupling
$$
\gamma=(\mathrm{Id},u)_\#\mu_0,\quad\text{for}\quad u(x)=Ax\quad\text{with}\quad A=\left(\begin{array}{cc}0&-1\\
1&0
\end{array}\right)
$$
for $\mu_0=\mathcal N(0,\mathrm{Id})$ in $d=2$ dimensions.
Then, it holds for $\gamma^{\otimes 2}$-almost every $(x_1,y_1)$, $ (x_2,y_2)$ that $y_1=Ax_1$ and $y_2=Ax_2$. In particular, we have
$\langle x_1-x_2,y_1-y_2\rangle=\langle x_1-x_2,A(x_1-x_2)\rangle=0$, so that:
$$
\|x_1-y_1\|^2+\|x_2-y_2\|^2=\|x_1-y_2\|^2+\|x_2-y_1\|^2.
$$
In particular, the mappings $S^{(1)}=\mathrm{Id}$ and $S^{(2)}$ defined by
$$
S^{(2)}((x_1,y_1),(x_2,y_2))=((x_1,y_2),(x_2,y_1))
$$
are both solutions of the discrete OT on $\mathrm{supp}(\gamma)=\Gr u$. However, if we define the minibatch OT with $S=S^{(1)}$, we obtain that $\mathcal F_N(\gamma)=\gamma$, while for $S=S^{(2)}$, we obtain that $\mathcal F_N(\gamma)=\mu_0\otimes\mu_1$.
\end{example}

\subsection{Properties of $N$-cyclically Monotone Plans}

Proposition~\ref{prop:minibatchOT} tells us that the minibatch OT operator $\mathcal F_N$ always reduces the transport cost unless the input plan is $N$-cyclically monotone. Therefore, in order to study limit points of minibatch OT, we first collect some properties of $N$-cyclically monotone couplings. In particular, we show that $N$-cyclically monotone couplings already have very useful properties in the context of generative modeling, like straightness and rectifiability. 

\begin{proposition}\label{prop:N-monotone}
Let $\gamma\in\Gamma(\mu_0,\mu_1)$ be $N$-cyclically monotone and denote the corresponding velocities by $v_t=\argmin_{w_t\in L^2(\mu_t)}\mathcal L(w_t|\gamma)$. Then, the following holds:
\begin{enumerate}[(i)]
    \item $\mathcal L(v_t|\gamma)=0$.
    \item $\gamma$ is its unique rectified coupling, i.e., $\mathcal R(\gamma)=\{\gamma\}$.
    \item If $\mu_0$ is absolutely continuous, then $\gamma$ is rectifiable and there exists a monotone map $u$ such that $\gamma=(\mathrm{Id},u)_\#\mu_0$.
\end{enumerate}
\end{proposition}
\begin{proof}
\begin{enumerate}[(i)]
    \item By definition, we have that $\mathrm{supp}(\gamma)=\mathrm{Gr}(u)$ for some set-valued monotone map $u$. We now define $w_t(x)=\frac{1}{1-t}\left((t\mathrm{Id}+(1-t)u^{-1})^{-1}(x)-x\right)$ and show that $\mathcal L(w_t|\gamma)=0$. Note that since $u^{-1}$ is a monotone set-valued map, $(t\mathrm{Id}+(1-t)u^{-1})^{-1}$ is single-valued. Now let $(x,y)\in\mathrm{Gr}(u)=\mathrm{supp}(\gamma)$ and $t\in(0,1)$. Then, we have that 
    $$
    w_t((1-t)x+ty)=\frac{1}{1-t}(t\mathrm{Id}+(1-t)u^{-1})^{-1}((1-t)x+ty)-x-\frac{t}{1-t}y.
    $$
    Since $ty+(1-t)x\in (t\mathrm{Id}+(1-t)u^{-1})(y)$ and since $(t\mathrm{Id}+(1-t)u^{-1})^{-1}$ is single-valued, we have that $y=(t\mathrm{Id}+(1-t)u^{-1})^{-1}((1-t)x+ty)$.
    In particular, we have that
    \[
    w_t((1-t)x+ty)=\frac{1}{1-t} y - x - \frac{t}{1-t} y=y-x.
    \]
    Because this implies that $w_t((1-t)x+ty)=y-x$ $\gamma$-almost everywhere, we obtain
    \[
    \mathcal L(w_t|\gamma)=\int_0^1\int \|w_t((1-t)x+ty)-y+x\|^2\d\gamma(x,y)\d t=0.
    \]
    Since the loss is always non-negative, this implies that $w_t$ 
    minimizes $\mathcal L(\cdot|\gamma)$ and that $w_t=v_t$.
    \item Note that $(t\mathrm{Id}+(1-t)u^{-1})^{-1}$ is up to a constant the resolvent of a monotone mapping and therefore Lipschitz continuous. Hence, we get by the proof of the previous part, that $v_t$ is Lipschitz continuous in space for $t\in(0,1)$ and continuous in time. Thus the trajectories $\lambda$ in the measure $\Lambda$ from Definition~\ref{def:weak_rec} (see also Corollary~\ref{cor:properties_weak_rec}) are uniquely defined for $t\in(0,1)$. Since (almost every) trajectory has finite length, also the limits $t=0$ and $t=1$ are uniquely defined.
    Checking that $\Lambda(A)=\gamma(\{(x,y):\lambda(t)=(1-t)x+ty \text{ fulfills }\lambda\in A\})$ fulfills the conditions \eqref{eq:Lambda1} and \eqref{eq:Lambda2} for Definition~\ref{def:weak_rec} implies that $\gamma\in\mathcal R(\gamma)$.
    \item By \cite[Thm 4.3]{CD2014}, we know that $\gamma=(\mathrm{Id},u)_\#\mu_0$ for a (not-necessarily maximal) monotone Borel map $u$, which implies that $v_0(x)=u(x)-x$. Moreover, we know by part (i) that $v_t((1-t)x+tu(x))=u(x)-x$ (part (i) only says that this identity holds a.e., however, monotonicity of $u$ implies that there does not exist any other $x'$ with $(1-t)x+tu(x)=(1-t)x'+tu(x')$, so that $v_t$ can be chosen such that the identity holds pointwise). Together we obtain that $t\mapsto (1-t)x+tu(x)$ solves the flow ODE. Due to part (ii) solutions are unique.    
\end{enumerate}
\end{proof}

Additionally, Caffarelli studies in \cite{Caffarelli97} the regularity of $2+\epsilon$ monotone maps, which includes $3$-monotone maps as a special case. In our context \cite[Thm 0.7]{Caffarelli97} implies that $\gamma=(\mathrm{Id},u)_\#\mu_0$ for a Hölder-continuous monotone invertible mapping $u\colon\R^d\to\R^d$ whenever $N\geq 3$ and $\mu_0$ and $\mu_1$ are both absolutely continuous with densities bounded from above and locally from below.

\subsection{Reflow with Minibatch OT}

Next, we consider the Reflow algorithm, where we apply minibatch OT in each iteration. Formally, this corresponds to generating a sequence of transport plans by Algorithm~\ref{alg:minibatch_OT_rec}. 
\begin{algorithm}[t]
\begin{algorithmic}
    \State Given: batch size $N$, initial coupling $\gamma_0$
    \For{$n=0,1,2,...$}
    \State $\gamma_n^{(N)}=\mathcal F_N(\gamma_n)$\Comment{Minibatch OT}
    \State $\gamma_{n+1}\in\mathcal R(\gamma_n^{(N)})$ \Comment{Reflow}
    \EndFor
\end{algorithmic}
\caption{Reflow with Minibatch OT}
\label{alg:minibatch_OT_rec}
\end{algorithm}
It is directly clear from Corollary~\ref{cor:properties_weak_rec} and Proposition~\ref{prop:minibatchOT} that it holds
\begin{equation}\label{eq:decreasing}
\int \|x-y\|^2\d\gamma_n(x,y)\geq\int \|x-y\|^2\d\gamma_n^{(N)}(x,y)\geq \int \|x-y\|^2\d\gamma_{n+1}(x,y).
\end{equation}

In the following, we analyze limit points of the sequences $(\gamma_n)_n$ and $(\gamma_n^{(N)})_n$. Note again that such limit points exist since $\Gamma(\mu_0,\mu_1)$ is relatively compact. Moreover, $\Gamma(\mu_0,\mu_1)$ has uniformly integrable $2$-moments, so that weak convergence and convergence in $(\mathcal P_2(\R^d\times\R^d),W_2)$ coincide.

\begin{theorem}\label{thm:monotone_limit}
Let $(\gamma_n)_n$ and $(\gamma_n^{(N)})_n$ be generated by Algorithm~\ref{alg:minibatch_OT_rec} with batch size $N$ and let $\gamma$ and $\gamma^{(N)}$ be limit points. 
Then, $\gamma$ is $N$-cyclically monotone and it holds
$$
\int\|x-y\|^2\d\gamma(x,y)=\int\|x-y\|^2\d\gamma^{(N)}(x,y).
$$
\end{theorem}
\begin{proof}
Without loss of generality, we assume that the whole sequence $(\gamma_n)_n$ converges. Otherwise, the same proof can be conducted by going over to a convergent subsequence.

By \eqref{eq:decreasing}, we obtain that there exists some $T\in\R$ such that:
$$
\lim_{n\to\infty}\int\|x-y\|^2\d\gamma_n(x,y)=\lim_{n\to\infty}\int\|x-y\|^2\d\gamma_n^{(N)}(x,y)=T.
$$
Since the transport cost $\eta\mapsto \int \|x-y\|^2\d\eta(x,y)$ is continuous wrt weak convergence, we obtain that
$$
\int\|x-y\|^2\d\gamma(x,y)=\lim_{n\to\infty}\int\|x-y\|^2\d\gamma_n(x,y).
$$
Further, we know by Proposition~\ref{prop:minibatchOT} (iv) that
\begin{align}
\int\|x-y\|^2\d\mathcal F_N(\gamma)(x,y)&=\lim_{n\to\infty}\int\|x-y\|^2\d\mathcal F_N(\gamma_n)(x,y)\\
&=\lim_{n\to\infty}\int\|x-y\|^2\d\gamma_n^{(N)}(x,y)=T.
\end{align}
In particular, we have that $\int\|x-y\|^2\d\gamma(x,y)=\int\|x-y\|^2\d\mathcal F_N(\gamma)(x,y)$ which implies by Proposition~\ref{prop:minibatchOT} (iii) that $\gamma$ is $N$-cyclically monotone.
\end{proof}

Together with the previous results, we obtain the following corollary which directly follows from Theorem~\ref{thm:monotone_limit} and Proposition~\ref{prop:N-monotone}.
\begin{corollary}\label{cor:limit_points_minibatch_reflow}
Let $(\gamma_n)_n$ be generated by Algorithm~\ref{alg:minibatch_OT_rec} and let $\gamma$ be a limit point. 
Then, if $\mu_0$ is absolutely continuous, we have that $\gamma$ is $N$-cyclically monotone, rectifiable and straight in the sense that it fulfills $\min_{w_t\in L^2(\mu_t)}\mathcal L(w_t|\gamma)=0$.
\end{corollary}

So far, we do not know whether the limits of $\gamma_n$ and $\gamma_n^{(N)}$ coincide. To this end, we consider the following proposition. The proof requires some technical effort and is therefore deferred to Appendix~\ref{app:limits_coincide}.

\begin{proposition}\label{prop:limits_coincide}
Assume that $\mu_0$ and $\mu_1$ are absolutely continuous and let $(\gamma_n)_n$ and $(\gamma^{(N)}_n)_n$ be generated by Algorithm~\ref{alg:minibatch_OT_rec}. Further assume that we have a subsequence $(n_k)_k$ such that $\gamma_{n_k}\wto \gamma$ for some $\gamma$. Then, it also holds for any $l\in\Z_{\geq 0}$ that $\gamma_{n_k+l}\wto \gamma$ and $\gamma_{n_k+l}^{(N)}\wto \gamma$. In particular the limit points of the sequences $(\gamma_n)_n$ and $(\gamma^{(N)}_n)_n$ coincide.
\end{proposition}

We remark that in general we cannot conclude convergence of the whole sequences $(\gamma_n)_n$ and $(\gamma^{(N)}_n)_n$ from the proposition.

\section{Reflow with Minibatch OT and Gradient Constraint}\label{sec:gradient}

In order to approach optimal transport with reflow, Liu \cite{L2022} proposed to restrict the velocity field in the flow matching loss \eqref{eq:flow_matching_loss} to be a gradient, i.e., to define
$$
v_t^p\in \argmin_{w_t\in L^2(\mu_t)} \mathcal L(w_t|\gamma)\quad\text{subject to } \quad w_t=\nabla \varphi \quad \text{ for some } \varphi\colon\R^d\to\R.
$$
In order to ensure existence of the above optimization problem, we have to relax the constraint $w_t=\nabla \varphi$ to $w_t\in\mathrm{T}_{\mu_t}$, where ${\mathrm T}_{\mu_t}\coloneqq\overline{\{\nabla \varphi:\varphi\in C_c^\infty(\R^d)\}}^{L^2(\mu_t)}$ is the $L^2$-closure of all gradients, see \cite[Prop 8]{HCD2025}. This space is also referred to as the (reduced) Wasserstein tangent space at $\mu_t$.
It is easy to check that a minimizer $v_t^p$ of the loss
in $\mathrm{T}_{\mu_t}$ will satisfy $\Div v_t^p\mu_t = \Div v_t\mu_t$,
so that the continuity equation is also satisfied with this
new vector field~\cite{L2022}.
Following the same idea, we can now define a weak rectified coupling with gradient constraint analogously to Definition~\ref{def:weak_rec}.

\begin{definition}[Weak Rectified Couplings with Gradient Constraint]\label{def:weak_rec_grad} 
Let $\gamma\in\Gamma(\mu_0,\mu_1)$ and $\mu_t$ be defined by \eqref{eq:def_mu_t} and denote by $v_t^p=\argmin_{w_t\in \mathrm{T}_{\mu_t}}\mathcal L(w_t|\gamma)$ the solution of the flow matching loss \eqref{eq:flow_matching_loss} with gradient constraint.
Then, we define the set of weak rectified couplings with gradient constraint $\mathcal R_p(\gamma)$ as the set of all transport plans $\tilde \gamma$ given by
\begin{equation}\label{eq:gamma_tilde2}
\int_{\R^d\times \R^d} \psi(x,y)  \d\tilde\gamma=  \int \psi(\lambda(0),\lambda(1)) \d\Lambda(\lambda),\quad\text{for all }\psi\in C_c^\infty(\R^d\times\R^d),
\end{equation}
where $\Lambda$ fulfills \eqref{eq:Lambda1} and \eqref{eq:Lambda2} with respect to $v_t^p$.
\end{definition}
It is straightforward to show that all properties of $\mathcal R(\gamma)$ from Corollary~\ref{cor:properties_weak_rec} also hold for $\mathcal R_p(\gamma)$.
Fixed points of this variant of reflow only coincide with the optimal transport under strong assumptions on the regularity of the velocity field and the support of the intermediate measures $\mu_t$, see \cite{L2022,HCD2025}. In this section, we consider the algorithm, which alternates between this reflow variant and minibatch OT, see Algorithm~\ref{alg:minibatch_OT_rec_grad}.
\begin{algorithm}[t]
\begin{algorithmic}
    \State Given: batch size $N$, initial coupling $\gamma_0$
    \For{$n=0,1,2,...$}
    \State $\gamma_n^{(N)}=\mathcal F_N(\gamma_n)$\Comment{Minibatch OT}
    \State $\gamma_{n+1}\in\mathcal R_p(\gamma_n^{(N)})$ \Comment{Reflow with gradient constraint}
    \EndFor
\end{algorithmic}
\caption{Reflow with Minibatch OT and Gradient Constraint}
\label{alg:minibatch_OT_rec_grad}
\end{algorithm}
We first study fixed points of this algorithm and afterwards derive some results on the convergence of the iterates.

\subsection{Fixed Points}

We start with investigating fixed points of $\mathcal R_p\circ\mathcal F_N$. The first lemma shows that a plan is a fixed point of this operator if and only if it is a fixed point of $\mathcal F_N$ and of $\mathcal R_p$.

\begin{lemma}\label{lem:fix}
Let $\gamma\in\Gamma(\mu_0,\mu_1)$. Then we have that $\gamma\in\mathcal R_p(\mathcal F_N(\gamma))$ if and only if $\gamma=\mathcal F_N(\gamma)$ and $\gamma\in\mathcal R_p(\gamma)$.
\end{lemma}
\begin{proof}
Let $\gamma^{(N)}=\mathcal F_N(\gamma)$. Since the Wasserstein transport cost is non-increasing for both $\mathcal F_N$ and $\mathcal R_p$ we get that
\[
\int\|x-y\|^2\d\gamma(x,y)=\int\|x-y\|^2\d\gamma^{(N)}(x,y).
\]
By Proposition~\ref{prop:minibatchOT} this implies that $\gamma$ is $N$-cyclically monotone.
Next, define $\mu_t^{(N)}$ and $\mu_t$ as in \eqref{eq:def_mu_t}
and let $v_t=\argmin_{w_t\in L^2(\mu_t^{(N)})}\mathcal L(w_t|\gamma^{(N)})$. Then, we obtain by \cite[Thm 5.3]{L2022} that $\int\|x-y\|^2\d\gamma(x,y)\leq \int\|x-y\|^2\d\gamma^{(N)}(x,y)-\mathcal L(v_t|\gamma^{(N)})$, so that $\mathcal L(v_t|\gamma^{(N)})=0$.
In particular, we know for $\Lambda$ from Definition~\ref{def:weak_rec} that $\lambda(t)=(1-t)\lambda(0)+t\lambda(1)$ and $\dot\lambda(t)=\lambda(1)-\lambda(0)$ for $\Lambda$-a.e.~$\lambda$. Hence, we get for all $\psi\in C_c^\infty(\R^d)$ that
\begin{align}
\int\psi(x)\d\mu_t(x)&=\int\psi((1-t)x+ty)\d\gamma(x,y)=\int \psi((1-t)\lambda(0)+t\lambda(1))\d\Lambda(\lambda)\\
&=\int \psi(\lambda(t))\d\Lambda(\lambda)=\int\psi(x)\d\mu_t^{(N)}(x),
\end{align}
so that $\mu_t=\mu_t^{(N)}$.
In addition, we have by Corollary~\ref{cor:properties_weak_rec} that
\begin{align}
\mathcal L(v_t|\gamma)&=\int_0^1\int \|v_t((1-t)x+ty)-y+x\|^2\d\gamma(x,y)\d t\\
&=\int_0^1\int \|v_t((1-t)\lambda(0)+t\lambda(1))-\lambda(1)+\lambda(0)\|^2\d\Lambda(\lambda)\d t\\
&=\int_0^1\int \|v_t(\lambda(t))-\dot \lambda(t)\|^2\d\Lambda(\lambda)\d t=0.
\end{align}
Defining the plans $\gamma_t$ by
\[
\int \psi(x,y)\d\gamma_t(x,y)=\int\psi((1-t)x+ty,y)\d\gamma(x,y)
\]
and $\gamma_t^{(N)}$ analogously, we obtain from $\mathcal L(v_t|\gamma)=\mathcal L(v_t|\gamma^{(N)})=0$ that for almost every $t$ it holds
\[
0=\int \|x+(1-t)v_t(x)-y\|^2\d\gamma_t(x,y)=\int \|x+(1-t)v_t(x)-y\|^2\d\gamma_t^{(N)}(x,y),
\]
so that both $\gamma_t$ and $\gamma_t^{(N)}$ are supported on $\Gr u_t$ with $u_t=I+(1-t)v_t$. Hence, we have for any measurable rectangle $A\times B$ that
\begin{align}
\gamma_t(A\times B)&=\gamma_t(\{(x,u_t(x)):x\in A, u_t(x)\in B\})\\
&=\gamma_t(\{x\in A:u_t(x)\in B\}\times \R^d)=\mu_t(\{x\in A:u_t(x)\in B\}).
\end{align}
By the same argument we have that $\gamma_t^{(N)}(A\times B)=\mu_t(\{x\in A:u_t(x)\in B\})$. Since $A$ and $B$ were chosen arbitrarily, this implies $\gamma_t=\gamma_t^{(N)}$.
Given that $\gamma_t\wto \gamma$ and $\gamma_t^{(N)}\wto \gamma^{(N)}$ this implies $\gamma=\gamma^{(N)}$.  
\end{proof}

Together with Proposition~\ref{prop:minibatchOT}, the lemma implies that any fixed point $\gamma$ of $\mathcal R_p\circ\mathcal F_N$ is $N$-cyclically  monotone and it holds $\min_{w_t\in\mathrm{T}_{\mu_t}}\mathcal L(w_t|\gamma)=0$, where $\mu_t$ are the interpolations from \eqref{eq:def_mu_t} with respect to $\gamma$.
Then, we obtain that fixed points of $\mathcal R_p\circ\mathcal F_N$ are optimal transport under some support condition by the following theorem.

\begin{theorem}\label{thm:optimality}
Let $\gamma\in\Gamma(\mu_0,\mu_1)$ for $\mu_0,\mu_1\in L^\infty(\R^d)$ be $N$-cyclically monotone and assume that it fulfills
$\min_{w_t\in\mathrm{T}_{\mu_t}}\mathcal L(w_t|\gamma)=0$ for the interpolations $\mu_t$ from \eqref{eq:def_mu_t}.
Further, assume that there exists some $t\in(0,1)$ such that $\mu_t$ is absolutely continuous and that $\mu_t^{-1}\in L^1_{\textup{loc}}(\R^d)$. 
Then, $\gamma$ is an optimal transport plan between $\mu_0$ and $\mu_1$.
\end{theorem}
\begin{remark}
From the proof, one sees that the assumption on $\mu_t^{-1}$ could
be slightly weakened: it is sufficient that $\mu_t$ is absolutely continuous and that there exists a connected open set $A\subset \R^d$ with $\{\mu_t >0\}=A$ and $\mu_t^{-1}\in L^1_{\textup{loc}}(A)$. (If $\mu_t$ is continuous, having $A=\{\mu_t>0\}$ connected is thus sufficient.)
Further, the proof already shows that $\mu_t$ is absolutely continuous and bounded from above.
\end{remark}
\begin{proof}
We know that there is a monotone
map $u$ such that $\gamma\mres \Gr u=\gamma$. We can assume, in fact,
that $u$ is maximal-monotone. 
From the assumption that $\mu_0,\mu_1\in L^\infty(\R^d)$, we also obtain that $\mu_t\in L^\infty(\R^d)$ by the following argument:
We have that $\mu_0=((1-t)\mathrm{Id}+tu)^{-1}_\#\mu_t$. Now $((1-t)\mathrm{Id}+tu)^{-1}$ is up to a constant the resolvent of a monotone map and therefore Lipschitz continuous. Hence, \cite[Thm 2.8]{EG2015}  implies that $\mu_t$ is absolutely continuous and that the density of $\mu_t$ is upper bounded by $L^d$ times the density of $\mu_0$, where $L$ is the Lipschitz constant of $((1-t)\mathrm{Id}+tu)^{-1}$.

Next,
for any $t,s\in [0,1]$, $\gamma_{t,s}$ defined as:
\[
  \int \psi(x,y)d\gamma_{t,s}(x,y)
  = \int \psi((1-t)x+t y,(1-s)x+s y)d\gamma(x,y)
\]
is a coupling between $\mu_t$ and $\mu_s$. Now, $\gamma_{t,s}$ is
supported on
\[
  \Gr u_{t,s} := \{ ((1-t)x+ty,(1-s)x+sy): y\in u(x)\},
\]
and if $(\xi,\eta), (\xi',\eta')\in \Gr u_{t,s}$ with $\xi=(1-t)x+ty$,
$\xi'=(1-t)x'+ty'$, etc, assuming first $s>t$ one uses
that $ \xi = \frac{t}{s}\eta  + \frac{s-t}{s}x$ to get
\[
  (\xi-\xi')\cdot(\eta-\eta ') \ge \frac{t}{s}|\eta-\eta'|^2 + \frac{s-t}{s} (x-x')\cdot(\eta-\eta') \ge \frac{t}{s}|\eta-\eta'|^2
\]
using that $u$ is monotone.
One deduces that $u_{t,s}$ is $\frac{s}{t}$-Lipschitz. Similarly, if
$s<t$, writing $\xi = \frac{1-t}{1-s}\eta + \frac{t-s}{1-s}y$ we
find that $u_{t,s}$ is $\frac{1-s}{1-t}$-Lipschitz.
Since by construction $u_{t,s}=u_{s,t}^{-1}$ we can equivalently
write  
that for $0<t,s<1$, $u_{t,s}$ is
 $\max\{\frac{1-s}{1-t},\frac{s}{t}\}$-Lipschitz
and $\min\{\frac{s}{t},\frac{1-s}{1-t}\}$-strongly monotone.

Hence, for any $t,s\in (0,1)$, $u_{t,s}$ is
a bi-Lipschitz homeomorphism and one has:
\[
  \mu_s = {u_{t,s}}_\# \mu_t.
\]
In addition, if for some $s\in (0,1)$, $\mu_{s}$ is absolutely
continuous, then all other $\mu_t$ are (for $t\in (0,1)$), with
a.e.:
\begin{equation}\label{eq:formulamus}
  \mu_s(x) = \mu_t(u_{s,t}(x))\det\nabla u_{s,t}(x).
\end{equation}
For $s>t$, one has $\det\nabla u_{t,s} \ge (\frac{1-s}{1-t})^{d}$
and $\det\nabla u_{s,t}\ge (\frac{t}{s})^{d}$ (so that, up to a set of
vanishing measure, $\{\mu_s>0\} = u_{t,s}(\{\mu_t>0\})$ and these sets
should have the same number of connected components).
Observe, more precisely that, given $t\in (0,1)$, if $A_t\subseteq\R^d$ is a
connected  open set
where $\mu_t$ is absolutely continuous and $\mu_t^{-1}$ is
integrable, then the same holds for  $\mu_s$ in $A_s:=u_{t,s}(A_t)$,
for any $s\in (0,1)$, since
\[
\int_{A_s} \frac{1}{\mu_s(x)}dx = \int_{A_t} \frac{\det\nabla u_{t,s}(x)}{\mu_s(u_{t,s}(x))}\,dx
= \int_{A_t} \frac{(\det\nabla u_{t,s}(x))^2}{\mu_t(x)}\,dx <+\infty.
\]

Then, for a fixed point one has: 
\[
  \int_{0}^{1} \int_{A_s} \|v_s (x) - (u_{s,1}(x)-u_{s,0}(x))\|^2d\mu_s ds = 0 
\]
for $v_s = \nabla \varphi_s$ with $\varphi_s\in W^{1,1}_{\text{loc}}(A_s)$
for a.e.~$s$, thanks to Lemma~\ref{lem:gradienttangent} in Appendix~\ref{sec:gradtan}.

Hence, one has $\nabla\varphi_s(x_s) = u_{s,1}(x_s)-u_{s,0}(x_s)$
for a.e. $s\in (0,1)$ and a.e.~$x_s\in A_s$,
with $x_s = (1-s)u_{s,0}(x_s) + s u_{s,1}(x_s)$.
Since $u_{s,0}$ is $\frac{1}{1-s}$-Lipschitz and $u_{s,1}$ is $\frac{1}{s}$-Lipschitz,
we find that
\[
  \frac{|\cdot|^2}{2(1-s)}+ \varphi_s\quad\text{ and }
  \quad  \frac{|\cdot|^2}{2s}- \varphi_s
\]
are convex in $A_s$, with Lipschitz gradient.
In particular, given $t\in (0,1)$ and $s$ with $t<s<1$,
for any $x_s \in A_s$ as long as $x_s-z\in A_t$,
\[
z\mapsto  \frac{|z|^2}{2(s-t)}+ \varphi_t(x_s-z)
\]
is strongly convex. Now, if $x_s=(1-s)x + sy$ with $y\in u(x)$,
and $z=(s-t)(y-x)$, then $x_s-z = (1-t)x+ty=:x_t$ and the gradient
of the above expression vanishes.
This means that
\[
  \varphi_{t,s} (x_s):= \min_{z:x_s-z\in A_t}  \frac{|z|^2}{2(s-t)}+ \varphi_t(x_s-z)
\]
is given by $(s-t)\frac{|x-y|^2}{2} + \varphi_t(x_t)$, in addition, a
standard first variation argument shows that $\nabla\varphi_{t,s}(x_s)
= \nabla\varphi_t(x_t) = y-x = z/(s-t)= (x_s-x_t)/(s-t)$.
It follows that $\nabla\varphi_{t,s}(x_s)=\nabla\varphi_s(x_s)$ in $A_s$,
and as a consequence, since these functions are defined up to a constant,
one can assume that for all $s>t$,
\begin{equation}\label{eq:infconvs}
  \varphi_{s} (x_s)=
  \min_{z:x_s-z\in A_t}  \frac{|z|^2}{2(s-t)}+ \varphi_t(x_s-z)
  =
  \min_{x_t\in A_t}  \frac{|x_s-x_t|^2}{2(s-t)}+ \varphi_t(x_t).
\end{equation}
In addition, we have $x_s = u_{t,s}(x_t) = x_t + (s-t)\nabla\varphi_t(x_t)$, and since this is the gradient
of a convex function, Brenier's theorem shows that $u_{t,s}$
is an optimal transport from $\mu_t\mres A_t$ to $\mu_s\mres A_s$.
In the limit, we also deduce that $\gamma^{A}=
\lim_{s\to 0} \gamma_{s,1-s}\mres (A_s\times A_{1-s})$
is an optimal coupling
between $\mu_0^{A_0}= \lim_{s\to 0} \mu_s\mres A_s$ and
$\mu_1^{A_1}= \lim_{s\to 1} \mu_s\mres A_s$, and in particular
if we can choose $A_s=\R^d$ (that is, if for some $s$,
$\mu_s^{-1}\in L^1_{\text{loc}}(\R^d)$), we have that $\gamma$
is an optimal coupling between $\mu_0$ and $\mu_1$.
\end{proof}

\subsection{Integrability of $\mu_t$}

The most restrictive assumption of the theorem is the integrability condition of $\mu_t$ since it depends on $\gamma$ (which is the fixed point and therefore analytically inaccessible). 
In the following, we consider the case that $\mu_0$ and $\mu_1=u_\#\mu_0$ are absolutely continuous and bounded from above and locally from below (i.e., for any $R>0$ there exists some $\delta$ such that $\mu_0\geq \delta$ and $\mu_1\geq\delta$ on $B_R(0)$). Further, we consider without loss of generality the case that $t=1/2$. 

If $u$ is the optimal transport $u=\nabla \varphi$ with convex $\varphi$, it satisfies the Monge-Amp\`ere equation
\begin{equation}\label{eq:Monge-Ampere}
\det(D^2\varphi)=f\coloneqq\frac{\mu_0}{\mu_1\circ\nabla\varphi}.
\end{equation}
Without further assumptions, Wang \cite[Example 1]{Wang95} shows that (in dimension $d=2$) there exist $\varphi$ and $\mu_1$ such that locally $\varphi\not\in W^{2,2}$ (note that Wang states this result for $W^{2,p}$ for $p>2$, but it can be extended to $p=2$ by a closer consideration). Noting that $\det(\frac{I+\nabla u}{2})\mu_{1/2}(\frac{\mathrm{Id}+u}{2})=\mu_0$, this implies for any ball $B$ and $A=(\frac{I+\nabla u}{2})^{-1}(B)$ that
\begin{align*}
\int_B\frac{1}{\mu_{1/2}(x)}\d x&=\int_A \frac{\det(\frac{I+\nabla u(x)}{2})}{\mu_{1/2}(\frac{x+u(x)}{2})}\d x\\&=\int_A \frac{\det(\frac{I+\nabla u(x)}{2})^2}{\mu_{0}(x)}\d x\geq \frac{1}{2^{2d}\|\mu_0\|_\infty}\int_A \det(I+\nabla u(x))^2\d x
\end{align*}
Since we have $\det(I+\nabla u(x))^2\geq\lambda_{\max}(\nabla u(x)^\tT \nabla u(x))=\|\nabla u(x)\|^2_{\text{op}}$, we obtain that $\int_B\frac{1}{\mu_{1/2}(x)}\d x\geq c\int_A\|\nabla u(x)\|^2_{\text{op}}\d x$ for some $c>0$. Because $\varphi\not\in W^{2,2}(\R^d)$, the right side is not locally integrable so that also $\mu_{1/2}^{-1}$ is not locally integrable.
So in summary, without further assumption on $\mu_0$ and $\mu_1$, $N$-monotonicity is not sufficient to ensure that the support condition is fulfilled.

On the other hand, if $u$ is the optimal transport and if the densities of $\mu_0$ and $\mu_1$ are additionally continuous, then also the ratio $f$ from \eqref{eq:Monge-Ampere} is continuous \cite{caffarelli1992regularity}.
In this case, Caffarelli \cite{Caffarelli90} showed that $\varphi\in W^{2,p}$ (locally) for any $p>1$. In particular, we get by a similar computation as above that
$$
\int_B\frac{1}{\mu_{1/2}(x)}\d x\leq \frac{1}{2^{2d}\min_A \mu_0(x)}\int_A (1+\|\nabla u(x)\|_\mathrm{op})^{2d}\d x<\infty.
$$
Consequently, at least the optimal transport fulfills the support condition provided that $\mu_0$ and $\mu_1$ are absolutely continuous with continuous density bounded from above and locally from below.

A relatively standard blow-up analysis shows that if $\xi$ is such that $\int_{B_r(\xi)}\frac{1}{\mu_{1/2}(x)}\d x = \infty$ for all $r>0$, then near $\bar x=u_{1/2,0}(\xi)$ there are blow-ups of $u$ (of the form $z\mapsto \frac{u(x+rz)-u(x)}{r}$ for $x\to\bar x$, $r\to 0$) which in some 2-dimensional subspace $(\R e_1,\R e_2)$ converge to the limit monotone map $\R e_1\times\{0\}\mapsto \{0\}\times \R e_2$. In 2D,
such a limit map would transport
a null set to a null set with intermediate states $x_t$ spanning
the whole of $\R^2$.

\subsection{Convergence of the Iterates}

Next, we consider the question under which conditions a limit point $\gamma$ of the iterates $(\gamma_n)_n$ from Algorithm~\ref{alg:minibatch_OT_rec_grad} is a fixed point of $\mathcal R_p\circ\mathcal F_N$. Replacing $\mathcal R$ by $\mathcal R_p$ in the proofs of Corollary~\ref{cor:limit_points_minibatch_reflow} and Proposition~\ref{prop:limits_coincide} shows that $\gamma$ is $N$-cyclically monotone, straight and rectifiable. In particular, we have that $\gamma$ is a fixed point of $\mathcal R\circ\mathcal F_N$. However, to ensure that $\gamma$ is also a fixed point of $\mathcal R_p\circ \mathcal F_N$, we would need that $v_t\in\argmin_{w_t\in L^2(\mu_t)}\mathcal L(w_t|\gamma)$ fulfills that $v_t\in\mathrm{T}_{\mu_t}$.
By construction, we additionally know that $v_t^n\in\argmin_{w_t\in \mathrm{T}_{\mu_{t,n}}}\mathcal L(w_t|\gamma_n)$ fulfills (up to subsequences) that $v_t^n\in L^2(\mu_{t,n})$ converges to $v_t\in L^2(\mu_t)$ in the sense of \cite[Def 5.4.3]{AGS2008}, see also Section~\ref{subsec:weak_reflow}.
However, the tangent spaces $\mathrm{T}_\mu$ are not closed under this notion of weak convergence and the assertion that this implies $v_t\in\mathrm{T}_{\mu_t}$ is false in general.
Nevertheless, as detailed in Lemma~\ref{lem:gradienttangent}, the gradient property passes to the limit if not only the measures $\mu_{t,n}$ converge weakly to $\mu_t$ but also $1/\mu_{t,n}$ converges weakly to $1/\mu_t$. Under this additional assumption, we obtain the following corollary.

\begin{corollary}
Let $(\gamma_n)_n$ be generated by Algorithm~\ref{alg:minibatch_OT_rec_grad} and $(n_k)_k$ a subsequence such that $(\gamma_{n_k})_k$ converges (weakly) to a limit point $\gamma\in\Gamma(\mu_0,\mu_1)$. Denote by $\mu_t$ and $\mu_{t,n}$ the corresponding interpolations \eqref{eq:def_mu_t}. Assume that there exists $t\in(0,1)$ such that $\mu_{t,n_k}$ is absolutely continuous for all $k$ and fulfills that $1/\mu_{t,n_k}\rightharpoonup 1/\mu_t$ 
weakly in $L^1_\mathrm{loc}(\R^d)$. Then $\gamma$ is an optimal transport plan between $\mu_0$ and $\mu_1$.
\end{corollary}
\begin{proof}
As derived above, we know that the limit $\gamma$ is $N$-cyclically monotone and that $v_t\in\argmin_{w_t\in L^2(\mu_t)}\mathcal L(w_t|\gamma)$ is the weak limit of $v_t^n\in\argmin_{w_t\in \mathrm{T}_{\mu_{t,n}}} \mathcal L(w_t|\gamma_n)$. Using Lemma~\ref{lem:gradienttangent} we find that $v_t$
coincides for a.e.~$t$ with the gradient of a function in $W^{1,1}_\mathrm{loc}(\R^d)$. Thus, the proof of Theorem~\ref{thm:optimality} can be reproduced with these assumptions, yielding that $\gamma$ is an optimal transport plan.
\end{proof}

\section{Example and Numerical Illustrations} \label{sec:numerics}

\begin{figure}
\centering
\begin{subfigure}{.25\textwidth}
\resizebox{\textwidth}{!}{
\begin{tikzpicture}[>=stealth, node distance=2cm]
    \def\r{2} 

    \foreach \i in {1,...,5} {
        \node[circle, fill=blue, inner sep=6pt] (N\i)
            at ({90 - (\i-1)*360/5}:\r) {};
    }

    \foreach \i/\j in {1/5, 2/1, 3/2, 4/3, 5/4} {
        \draw[->, thick, shorten >=3pt, shorten <=3pt] (N\i) -- (N\j);
    }
\end{tikzpicture}}
\caption*{$M=5$}
\end{subfigure}
\hspace{2cm}
\begin{subfigure}{.25\textwidth}
\resizebox{\textwidth}{!}{
\begin{tikzpicture}[>=stealth, node distance=2cm]
    \def\r{2} 

    \foreach \i in {1,...,7} {
        \node[circle, fill=blue, inner sep=6pt] (N\i)
            at ({90 - (\i-1)*360/7}:\r) {};
    }

    \foreach \i/\j in {1/7, 2/1, 3/2, 4/3, 5/4, 6/5, 7/6} {
        \draw[->, thick, shorten >=3pt, shorten <=3pt] (N\i) -- (N\j);
    }
\end{tikzpicture}}
\caption*{$M=7$}
\end{subfigure}
\caption{Illustration of the example in Section~\ref{sec:numerics} for $M=5$ and $M=7$ circles. The blue circles represent the measure $\mu_0=\mu_1$ defined in \eqref{eq:drehdich}, and the arrows show the mapping $u$ which carries the coupling $\gamma$. We find that for $M=2N+1$ the coupling $\gamma$ is $N$-, but not $N+1$-cyclically monotone. In particular, the coupling on the left is $2$- but not $3$-cyclically monotone and the coupling on the right is $3$- but not $4$-cyclically monotone.}
\label{fig:drehdich_malen}
\end{figure}
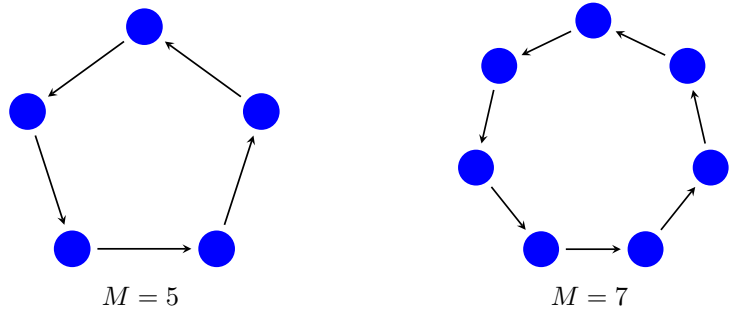

In this section, we provide an example which shows that if the support condition of Theorem~\ref{thm:optimality} is violated, there may  exist again fixed points of $\mathcal R_p\circ \mathcal F_N$ which do not coincide with the optimal transport.

To this end let $\eta$ be an absolutely continuous probability measure with smooth density with support in $B_\epsilon(0)$ and let $\eta_x$ be the translation of $\eta$ by the vector $x$.
Then, for some $M\in \Z_{\geq 5}$, we consider 
\begin{equation}\label{eq:drehdich}
\mu_0=\mu_1=\frac1M\sum_{k=1}^M \eta_{x_k}, \quad x_k=\Big(\cos\Big(\frac{2k\pi}{M}\Big),\sin\Big(\frac{2k\pi}{M}\Big)\Big).
\end{equation}
By definition, the optimal transport plan between $\mu_0$ and $\mu_1$ is carried by the identity map.
Here we consider the plan $(\mathrm{Id},u)_\#\mu_0$ with $u(x)=x+x_{k+1}-x_k$ for all $x\in B_\epsilon(x_k)$, see Figure~\ref{fig:drehdich_malen} for an illustration of $\mu_0=\mu_1$ and the coupling $\gamma$.
Since $u$ is (locally) just a translation, we get that $\gamma$ fulfills $\min_{w_t\in \mathrm{T}_{\mu_t}}\mathcal L(w_t|\gamma)=0$ which implies that $\gamma\in\mathcal R_p(\gamma)$. In addition, we show in Lemma~\ref{lem:drehdich_ist_monoton} of Appendix~\ref{app:wrong_with_disconnected_support} that for $M=2N+1$ the coupling $\gamma$ is $N$-cyclically monotone. By considering the points $x_{2k+1}$ for $k=0,1,...,N$ we can also see that this coupling is not $N+1$-cyclically monotone.

\paragraph{Numerical Simulation}
We illustrate the behavior of reflow with minibatch OT numerically. To this end, we set $\gamma_0$ to the coupling $\gamma$ as described above for $M=5$ and $M=7$ and simulate Algorithm~\ref{alg:minibatch_OT_rec} for different batch sizes $N$ in the minibatch OT.
In Figures~\ref{fig:numerik1} and \ref{fig:numerik2}, we plot samples from $\mu_0$ and $\mu_1$ as well as the trajectories of the flow ODE for the learned velocity fields $v_t^{n}=\argmin_{w_t\in L^2(\mu_{t,n})}\mathcal L(w_t|\gamma_{n})$ for $n=0,...,5$.
As expected we observe for $N=2$ and $M=5$ as well as for $N=3$ and $M=7$ that the coupling $\gamma_0$ is already a fixed point of $\mathcal R\circ \mathcal F_N$. For larger $N$, we can see that the iterates move towards the identity plan which is also the optimal transport. This convergence can be very slow for the minimal batch sizes $N=3$ (for $M=5$) and $N=4$ (for $M=7$) but becomes faster for higher choices of $N$. For very large batch sizes like $N=256$, we arrive close to the optimal coupling after two steps.

\begin{figure}[p]
\begin{figure}[H]
\def\figwidth{0.13}
    \centering

\begin{subfigure}[t]{.02\textwidth}
        \hfill\rotatebox{90}{\scriptsize\hspace{.3cm} $N=2$}
    \end{subfigure}\hfill
    \begin{subfigure}[t]{\figwidth\textwidth}
    \includegraphics[width=\textwidth]{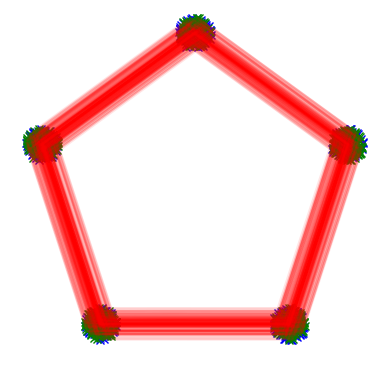}
    \end{subfigure}\hfill
    \begin{subfigure}[t]{\figwidth\textwidth}
    \includegraphics[width=\textwidth]{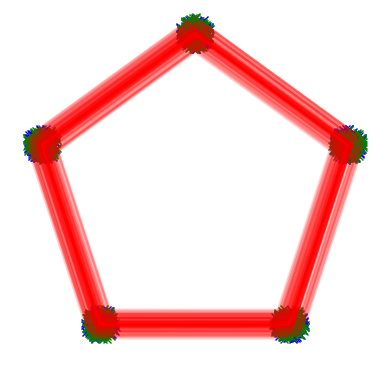}
    \end{subfigure}\hfill
    \begin{subfigure}[t]{\figwidth\textwidth}
    \includegraphics[width=\textwidth]{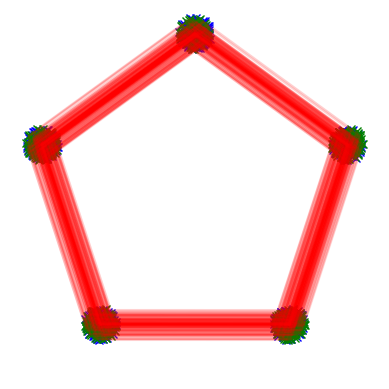}
    \end{subfigure}\hfill
    \begin{subfigure}[t]{\figwidth\textwidth}
    \includegraphics[width=\textwidth]{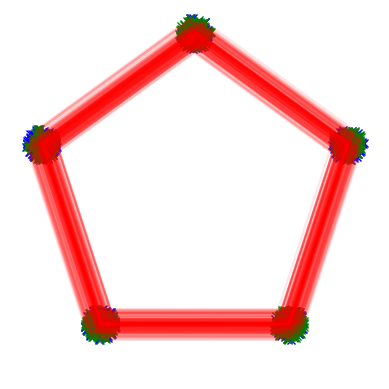}
    \end{subfigure}\hfill
    \begin{subfigure}[t]{\figwidth\textwidth}
    \includegraphics[width=\textwidth]{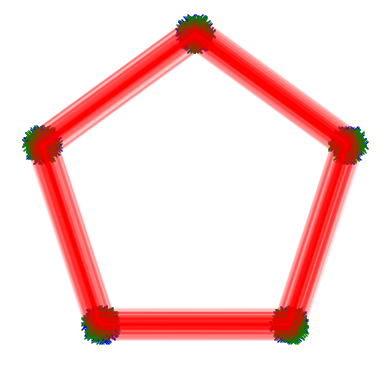}
    \end{subfigure}\hfill
    \begin{subfigure}[t]{\figwidth\textwidth}
    \includegraphics[width=\textwidth]{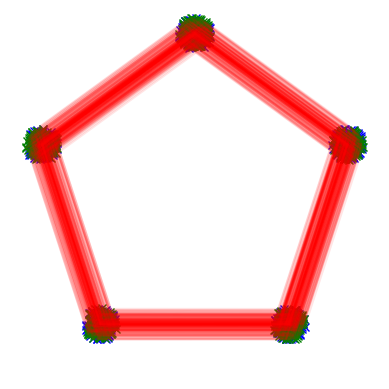}
    \end{subfigure}

\begin{subfigure}[t]{.02\textwidth}
        \hfill\rotatebox{90}{\scriptsize\hspace{.3cm} $N=3$}
    \end{subfigure}\hfill
    \begin{subfigure}[t]{\figwidth\textwidth}
    \includegraphics[width=\textwidth]{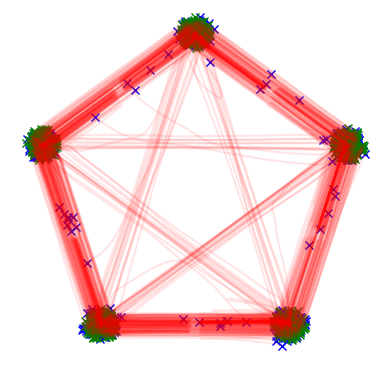}
    \end{subfigure}\hfill
    \begin{subfigure}[t]{\figwidth\textwidth}
    \includegraphics[width=\textwidth]{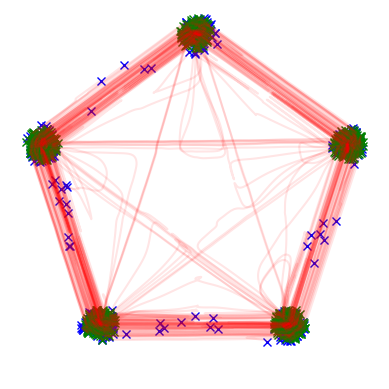}
    \end{subfigure}\hfill
    \begin{subfigure}[t]{\figwidth\textwidth}
    \includegraphics[width=\textwidth]{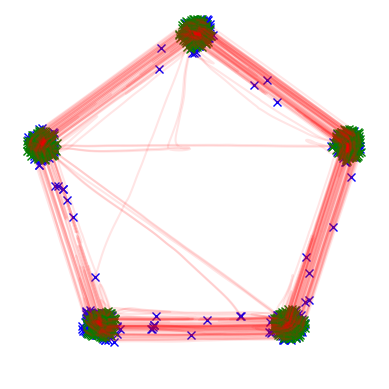}
    \end{subfigure}\hfill
    \begin{subfigure}[t]{\figwidth\textwidth}
    \includegraphics[width=\textwidth]{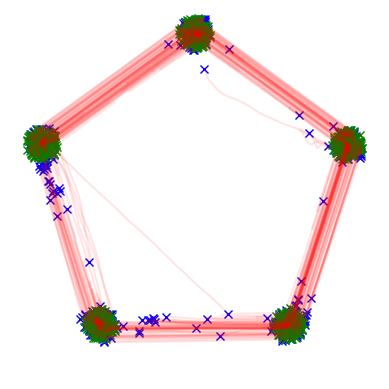}
    \end{subfigure}\hfill
    \begin{subfigure}[t]{\figwidth\textwidth}
    \includegraphics[width=\textwidth]{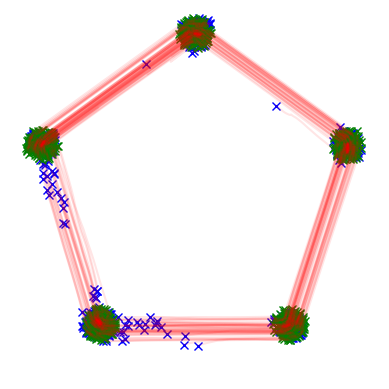}
    \end{subfigure}\hfill
    \begin{subfigure}[t]{\figwidth\textwidth}
    \includegraphics[width=\textwidth]{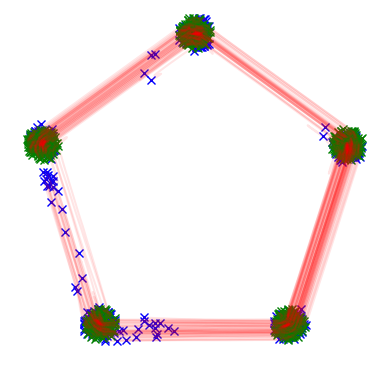}
    \end{subfigure}

        \begin{subfigure}[t]{.02\textwidth}
        \hfill\rotatebox{90}{\scriptsize\hspace{.3cm} $N=8$}
    \end{subfigure}\hfill
    \begin{subfigure}[t]{\figwidth\textwidth}
    \includegraphics[width=\textwidth]{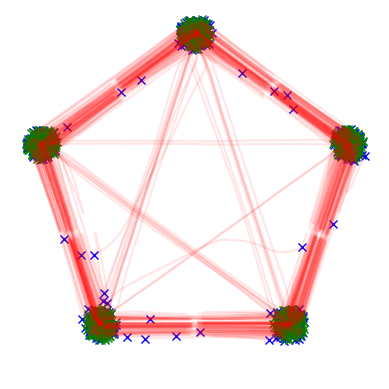}
    \end{subfigure}\hfill
    \begin{subfigure}[t]{\figwidth\textwidth}
    \includegraphics[width=\textwidth]{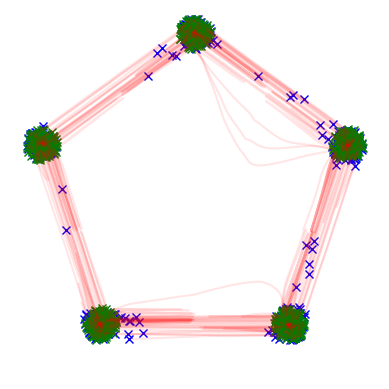}
    \end{subfigure}\hfill
    \begin{subfigure}[t]{\figwidth\textwidth}
    \includegraphics[width=\textwidth]{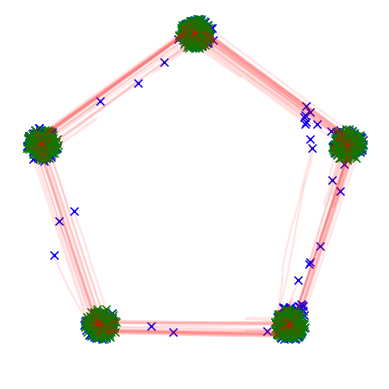}
    \end{subfigure}\hfill
    \begin{subfigure}[t]{\figwidth\textwidth}
    \includegraphics[width=\textwidth]{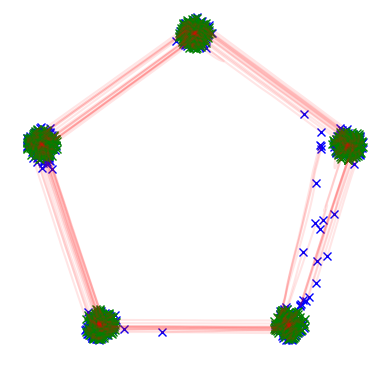}
    \end{subfigure}\hfill
    \begin{subfigure}[t]{\figwidth\textwidth}
    \includegraphics[width=\textwidth]{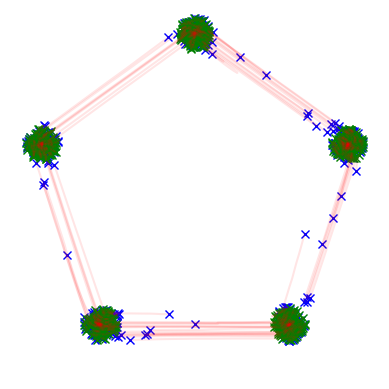}
    \end{subfigure}\hfill
    \begin{subfigure}[t]{\figwidth\textwidth}
    \includegraphics[width=\textwidth]{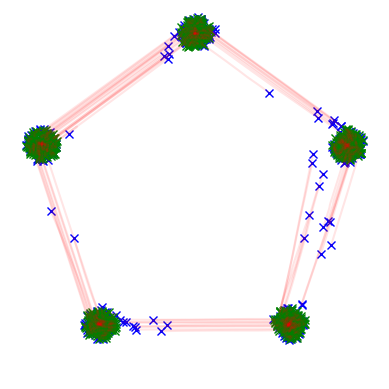}
    \end{subfigure}

    \begin{subfigure}[t]{.02\textwidth}
        \hfill\rotatebox{90}{\scriptsize\hspace{.15cm} $N=256$}
    \end{subfigure}\hfill
    \begin{subfigure}[t]{\figwidth\textwidth}
    \includegraphics[width=\textwidth]{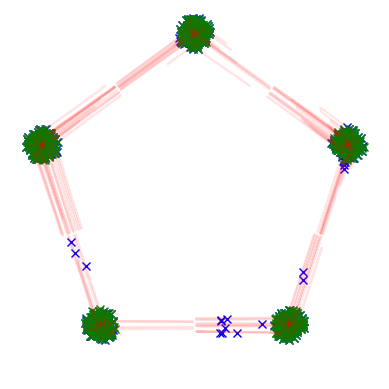}
    \end{subfigure}\hfill
    \begin{subfigure}[t]{\figwidth\textwidth}
    \includegraphics[width=\textwidth]{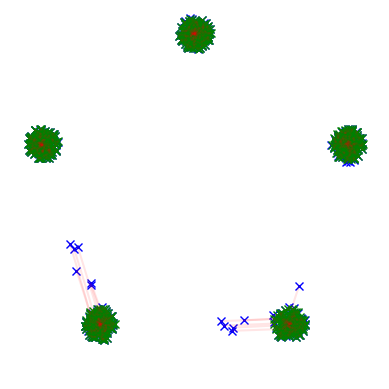}
    \end{subfigure}\hfill
    \begin{subfigure}[t]{\figwidth\textwidth}
    \includegraphics[width=\textwidth]{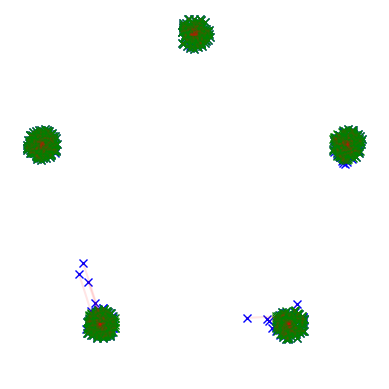}
    \end{subfigure}\hfill
    \begin{subfigure}[t]{\figwidth\textwidth}
    \includegraphics[width=\textwidth]{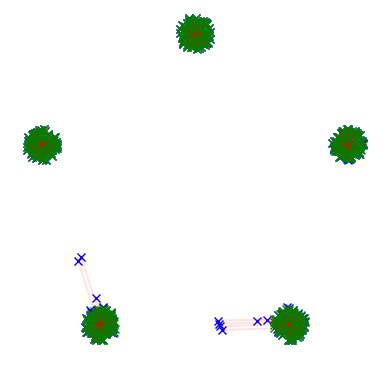}
    \end{subfigure}\hfill
    \begin{subfigure}[t]{\figwidth\textwidth}
    \includegraphics[width=\textwidth]{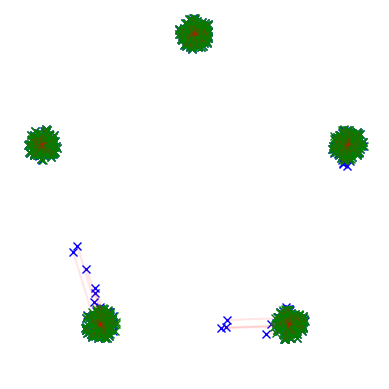}
    \end{subfigure}\hfill
    \begin{subfigure}[t]{\figwidth\textwidth}
    \includegraphics[width=\textwidth]{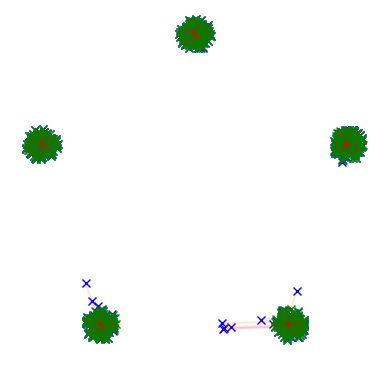}
    \end{subfigure}

\caption{We plot samples from $\mu_0$ and $\mu_1$ as well as trajectories of the flow ODE for $n=0,...,5$ generated by Algorithm~\ref{alg:minibatch_OT_rec} for the coupling $\gamma_0=\gamma$ as described in Section~\ref{sec:numerics} with $M=5$ modes. As expected, $\gamma_0$ is a fixed point of $\mathcal R\circ\mathcal F_N$ for $N=2$, but not for larger $N$.}
\label{fig:numerik1}
\end{figure}

\begin{figure}[H]
\def\figwidth{0.13}
    \centering

\begin{subfigure}[t]{.02\textwidth}
        \hfill\rotatebox{90}{\scriptsize\hspace{.3cm} $N=3$}
    \end{subfigure}\hfill
    \begin{subfigure}[t]{\figwidth\textwidth}
    \includegraphics[width=\textwidth]{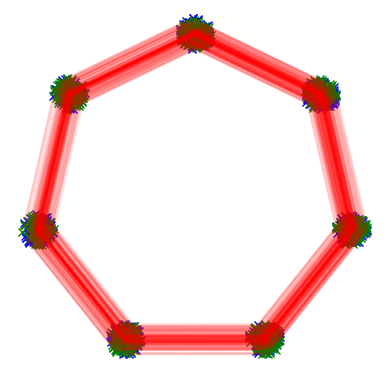}
    \end{subfigure}\hfill
    \begin{subfigure}[t]{\figwidth\textwidth}
    \includegraphics[width=\textwidth]{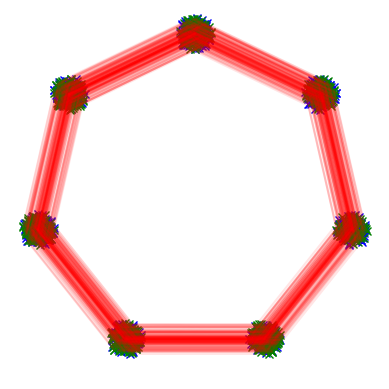}
    \end{subfigure}\hfill
    \begin{subfigure}[t]{\figwidth\textwidth}
    \includegraphics[width=\textwidth]{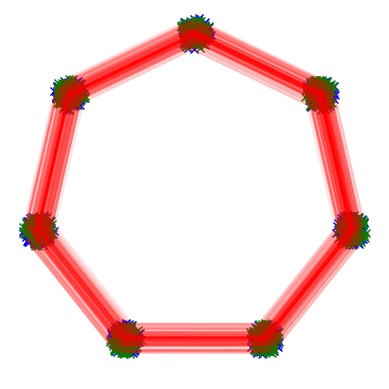}
    \end{subfigure}\hfill
    \begin{subfigure}[t]{\figwidth\textwidth}
    \includegraphics[width=\textwidth]{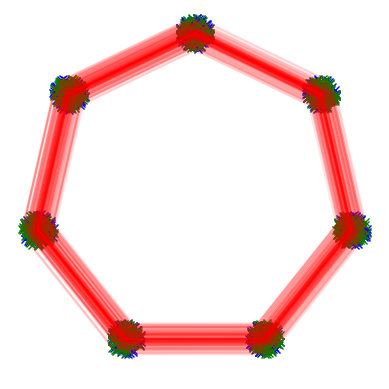}
    \end{subfigure}\hfill
    \begin{subfigure}[t]{\figwidth\textwidth}
    \includegraphics[width=\textwidth]{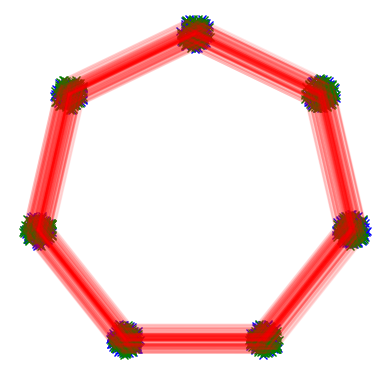}
    \end{subfigure}\hfill
    \begin{subfigure}[t]{\figwidth\textwidth}
    \includegraphics[width=\textwidth]{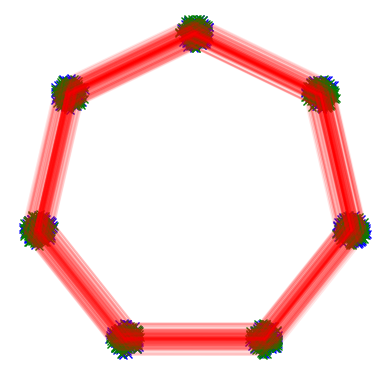}
    \end{subfigure}

    \begin{subfigure}[t]{.02\textwidth}
        \hfill\rotatebox{90}{\scriptsize\hspace{.3cm} $N=4$}
    \end{subfigure}\hfill
    \begin{subfigure}[t]{\figwidth\textwidth}
    \includegraphics[width=\textwidth]{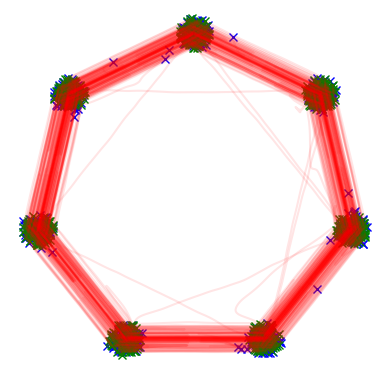}
    \end{subfigure}\hfill
    \begin{subfigure}[t]{\figwidth\textwidth}
    \includegraphics[width=\textwidth]{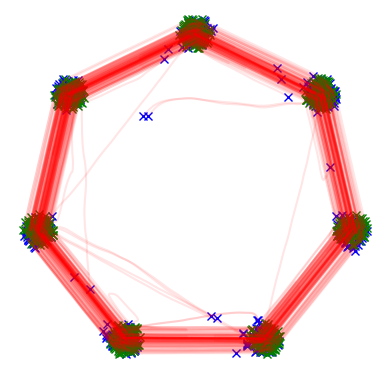}
    \end{subfigure}\hfill
    \begin{subfigure}[t]{\figwidth\textwidth}
    \includegraphics[width=\textwidth]{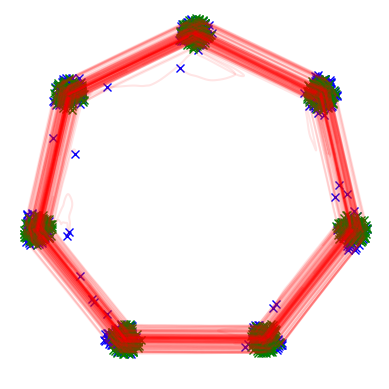}
    \end{subfigure}\hfill
    \begin{subfigure}[t]{\figwidth\textwidth}
    \includegraphics[width=\textwidth]{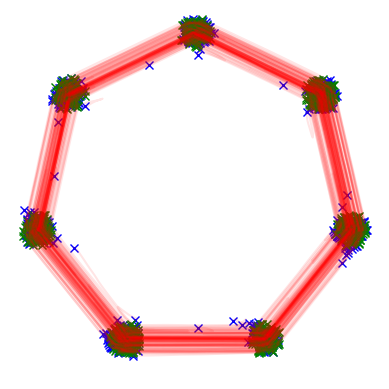}
    \end{subfigure}\hfill
    \begin{subfigure}[t]{\figwidth\textwidth}
    \includegraphics[width=\textwidth]{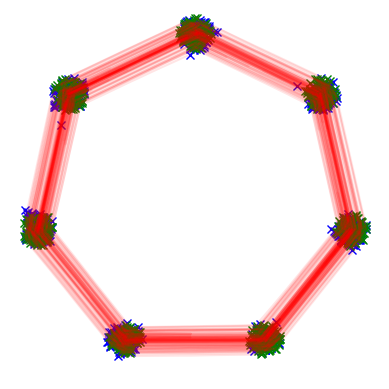}
    \end{subfigure}\hfill
    \begin{subfigure}[t]{\figwidth\textwidth}
    \includegraphics[width=\textwidth]{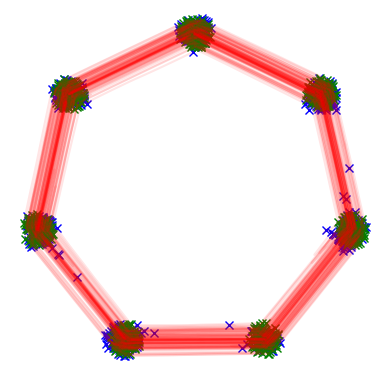}
    \end{subfigure}

        \begin{subfigure}[t]{.02\textwidth}
        \hfill\rotatebox{90}{\scriptsize\hspace{.3cm} $N=8$}
    \end{subfigure}\hfill
    \begin{subfigure}[t]{\figwidth\textwidth}
    \includegraphics[width=\textwidth]{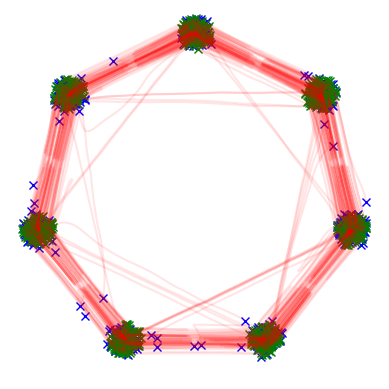}
    \end{subfigure}\hfill
    \begin{subfigure}[t]{\figwidth\textwidth}
    \includegraphics[width=\textwidth]{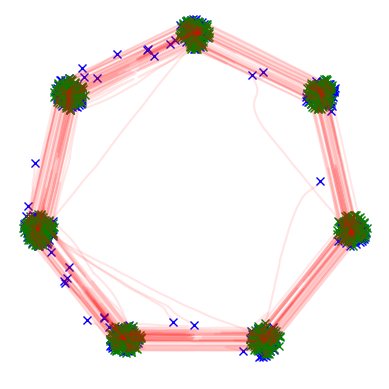}
    \end{subfigure}\hfill
    \begin{subfigure}[t]{\figwidth\textwidth}
    \includegraphics[width=\textwidth]{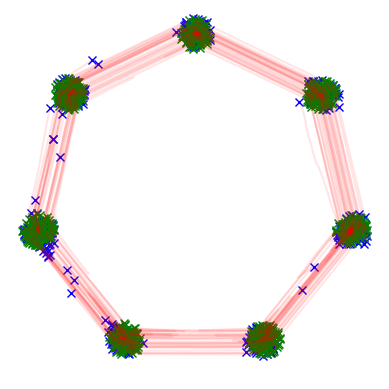}
    \end{subfigure}\hfill
    \begin{subfigure}[t]{\figwidth\textwidth}
    \includegraphics[width=\textwidth]{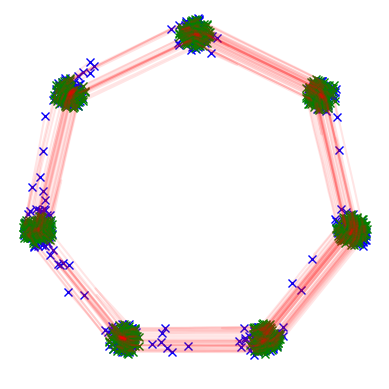}
    \end{subfigure}\hfill
    \begin{subfigure}[t]{\figwidth\textwidth}
    \includegraphics[width=\textwidth]{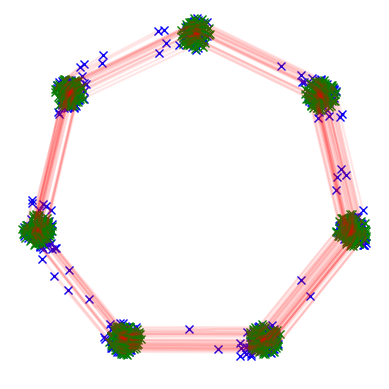}
    \end{subfigure}\hfill
    \begin{subfigure}[t]{\figwidth\textwidth}
    \includegraphics[width=\textwidth]{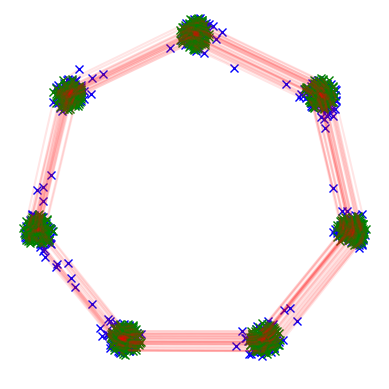}
    \end{subfigure}

    \begin{subfigure}[t]{.02\textwidth}
        \hfill\rotatebox{90}{\scriptsize\hspace{.15cm} $N=256$}
    \end{subfigure}\hfill
    \begin{subfigure}[t]{\figwidth\textwidth}
    \includegraphics[width=\textwidth]{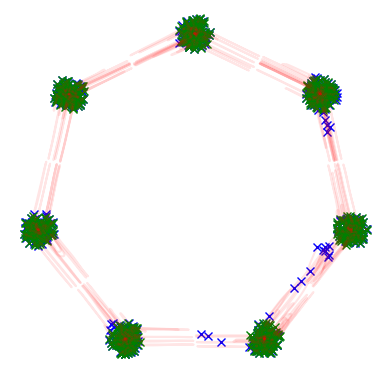}
    \end{subfigure}\hfill
    \begin{subfigure}[t]{\figwidth\textwidth}
    \includegraphics[width=\textwidth]{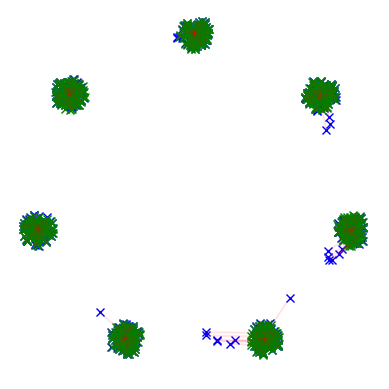}
    \end{subfigure}\hfill
    \begin{subfigure}[t]{\figwidth\textwidth}
    \includegraphics[width=\textwidth]{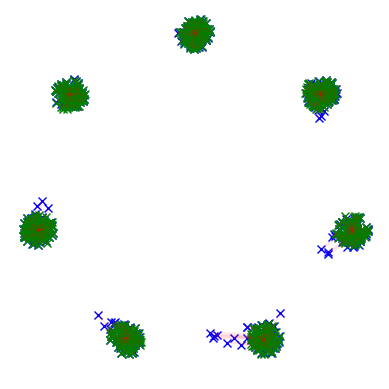}
    \end{subfigure}\hfill
    \begin{subfigure}[t]{\figwidth\textwidth}
    \includegraphics[width=\textwidth]{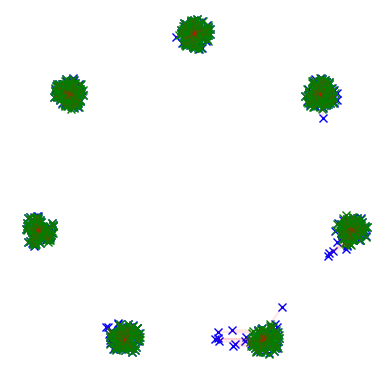}
    \end{subfigure}\hfill
    \begin{subfigure}[t]{\figwidth\textwidth}
    \includegraphics[width=\textwidth]{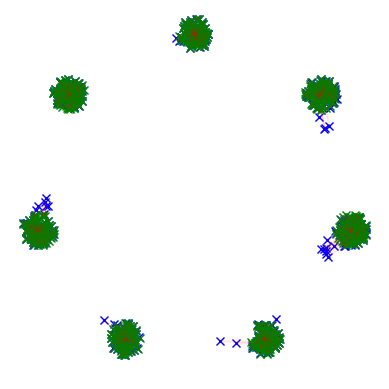}
    \end{subfigure}\hfill
    \begin{subfigure}[t]{\figwidth\textwidth}
    \includegraphics[width=\textwidth]{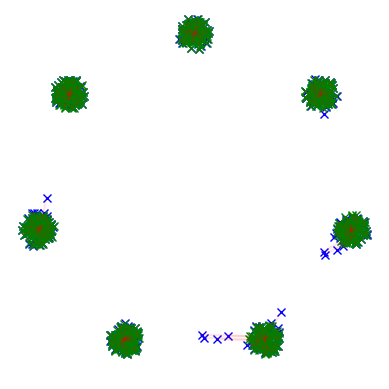}
    \end{subfigure}

\caption{We plot samples from $\mu_0$ and $\mu_1$ as well as trajectories of the flow ODE for $n=0,...,5$ generated by Algorithm~\ref{alg:minibatch_OT_rec} for the coupling $\gamma_0=\gamma$ as described in Section~\ref{sec:numerics} with $M=7$ modes. As expected, $\gamma_0$ is a fixed point of $\mathcal R\circ\mathcal F_N$ for $N=3$, but not for larger $N$. For $N=4$, the coupling changes very slowly, but $\gamma_0$ is not exactly a fixed point.}
\label{fig:numerik2}
\end{figure}
\end{figure}

\section{Conclusions}\label{sec:conclusions}

We investigated reflow and its limit points. First, we built weak rectified couplings to guarantee that the corresponding iteration is well-defined at every step.  
In combination with minibatch OT with batch size $N$ at each iteration, the limit points of reflow become $N$-cyclically monotone, which ensures both straightness and rectifiability (assuming $\mu_0$ is absolutely continuous).  
Further, we examined reflow with minibatch OT, where the velocities are additionally constrained to be gradients. With an additional support assumption, this procedure yields the optimal transport plan as the unique limit point. However, verifying this support condition is likely to be infeasible in practice and if it fails, the equivalence with optimal transport may no longer hold.

\paragraph{Open Questions, Future Work and Limitations}

In this work, we always assume that we find the exact minimizer of the flow matching loss \eqref{eq:flow_matching_loss}. A thorough investigation of how the results are influenced by approximation errors, introduced via the neural network parametrization and optimization, is left to future work. Further, we considered the reflow iterations starting from an arbitrary initial coupling $\gamma_0$. In most practical applications, this initial coupling is chosen as $\gamma_0=\mu_0\otimes\mu_1$ in which case the reflow might admit more regularity.\smallskip

Considering our specific analysis, there remain some questions. For reflow with minibatch OT, we characterized limit points and showed that they admit certain desirable properties. However, while it is clear from compactness that such limit points always exist, it is less clear whether they are unique or not. In particular for the case with minibatch OT Proposition~\ref{prop:limits_coincide} suggests that limit points might be unique in many cases, even though it does not imply this result yet.\smallskip

Finally, for reflow with minibatch OT and gradient constraint, it remains an open problem to find practically verifiable sufficient conditions that guarantee the support condition always holds. As a first step, one could attempt to generalize the results of \cite{Caffarelli97} to show that $N$-monotone plans become sufficiently regular when $N$ is large. However, this leads to involved analytical difficulties that lie beyond the scope of the present work.

\appendix

\section{A Remark on the Tangent Space}\label{sec:gradtan}

We give a simple condition under which elements of the Wasserstein tangent space
to a measure are gradients.

\begin{lemma} \label{lem:gradienttangent}
Let $B$ be an open ball, $\mu \in L^1(B)$ non-negative and
  assume $1/\mu \in L^1(B)$. Then any $v\in \mathrm{T}_{\mu}$ is a gradient in $B$.
  In addition, consider a sequence $(\mu_n)_{n\ge 1}$ of non-negative
  functions, with both $\mu_n\wto \mu$  and
$1/\mu_n\wto 1/\mu$ weakly
  in $L^1(B)$, as $n\to\infty$. 
  Then given for each $n\ge 1$, $v_n\in \mathrm{T}_{\mu_n}$ with $\sup_{n}\int |v_n|^2\mu_n\,dx<+\infty$,  up to subsequences,
  there exists 
  $\varphi  \in W^{1,1}(B)$ such that
  $v_n\stackrel{n\to\infty}{\rightharpoonup} D\varphi$ in $L^1(B;\R^d)$
  and $v_n\mu_n\wtos D\varphi\mu$  (weak-$*$ as measures).
\end{lemma}
\begin{remark}
It is not clear that the assumption of the second part are enough to
ensure that $D\varphi \in \mathrm{T}_\mu$ (in $B$), that is, that
it is also the strong $L^2(\mu)$ limit of gradients of smooth functions.
\end{remark}
\begin{proof}
 First, observe that by Cauchy-Schwartz, $L^2(B,\mu)\subset L^1(B)$,
 with continuous inclusion: if $v\in L^2(B,\mu;\R^d)$,
 \[
 \|v\|_{L^1(B)}=\int_B |v|\,dx \le \left( \int_B |v|^2\mu\,dx
 \int \frac{dx}{\mu} \right)^{\frac{1}{2}} = \sqrt{\int \frac{dx}{\mu}}\|v\|_{L^2(\mu)}.
 \]
 Then, if $v\in \mathrm{T}_\mu$ there exists $\varphi_n\in C_c^\infty(B)$ such that $D\varphi_n\to v$ in $L^2(B,\mu; \R^d)$, which implies that $D\varphi_n\to v $ in $L^1(B;\R^d)$. In particular (possibly
 replacing $\varphi_n$ with $\varphi_n-(1/|B|)\int_B\varphi_n \,dx$)
 there exists $\varphi\in W^{1,1}(B)$ with $\varphi_{n}\to\varphi$
 and $v=D\varphi$.\smallskip

Now, assume $v_n \in \mathrm{T}_{\mu_n}$ with $\mu_n\wto \mu$
in $L^1(B)$ and $\sup_n \int_B |v_n|^2 \mu_n \,dx<+\infty$. Observe that as soon as $\sup_n\int 1/\mu_n \,dx<+\infty$,
then $\mu_n = D\varphi_n$ for some $\varphi_n\in W^{1,1}(B)$ (uniformly
bounded) and
in addition, $\int_B 1/\mu\, dx\le \liminf_n \int_B 1/\mu_n\,dx<+\infty$.
Yet, this allows only, it seems, to assert that $(\varphi_n)_n$ is
precompact in $BV(B)$, and it is unclear that weak-$*$ limits of $D\varphi_n \mu_n$ (as measures) are of the form $v\mu$ with
either $v\in \mathrm{T}_\mu$ (which by the previous result would
imply that it is a gradient) or $v$ is related to the limit points of $D\varphi_n$.

If we assume additionally that $1/\mu_n$ is weakly precompact in 
$L^1$, we can show that also $D\varphi_n$ is. Indeed,
there exists $\Psi:\R_+\to \R_+$
  a (smooth), convex, increasing and superlinear function such that
  \begin{equation}
  \label{eq:Psibound}
   \sup_{n\ge 1} \int_B \Psi\left(\frac{1}{\mu_n}\right) dx <+\infty,
  \end{equation}
  (see for instance~\cite[Prop.~1.27]{AmbrosioFuscoPallara}).
  Possibly substituting $\Psi$ with $\Psi-\Psi(0)$ we  may assume $\min\Psi=0$.
  Letting $\tilde\Psi(t):=t\Psi(t)$,
  one finds a convex increasing function such that
  $\lim_{t\to+\infty}\tilde\Psi(t)/t^2=+\infty$.
  In addition, $\min \tilde\Psi=0$, $\tilde\Psi'(0)=0$, $\tilde\Psi'(t)=
  t\Psi'(t)+\Psi(t)>0$ for $t>0$. It follows that
  $\tilde\Psi^*\ge 0$, $\tilde\Psi^*(0) = 0$ and $\tilde\Psi^*(s)>0$ for $s>0$.
  
  If $\Phi$ is a nonnegative function, one writes:
  \[
    \int_B \Phi(|D\varphi_n|)dx =     \int_B \Phi(|D\varphi_n|)\frac{\mu_n}{\mu_n}dx
    \le \int_B \tilde\Psi^*\left(\Phi(|D\varphi_n|)\right)\mu_n \,dx
    + \int_B \tilde\Psi\left(\frac{1}{\mu_n}\right) \mu_n \,dx,
  \]
  where the latter integral is precisely $\int_B \Psi(1/\mu_n)\,dx$,
  bounded by~\eqref{eq:Psibound}.
  We define $\Phi$ so that for any $t\ge 0$,
  \[
    \tilde\Psi^*(\Phi(t)) = t^2 \Leftrightarrow \Phi(t) = (\tilde\Psi^*)^{-1}(t^2),
  \]
  then $\Phi$ is positive for positive $t$, increasing and goes to $+\infty$ at infinity.
  In this way,
  \[
    \int_B \Phi(|D\varphi_n|)dx \le
    \int_B |D\varphi_n|^2\mu_n \,dx
    + \int_B \Psi\left(\frac{1}{\mu_n}\right) dx 
  \]
  is uniformly bounded.
  Since $\lim_{t\to\infty}\tilde\Psi(t)/t^2=+\infty$,
  $\lim_{s\to\infty}\tilde\Psi^*(s)/s^2= 0$. Indeed for any
  $M>0$ there is $t_0$ such that for all $t$,
  $\tilde\Psi(t)\ge M t^2\chi_{[t_0,+\infty)}(t)$, and
  \[
    \tilde\Psi^*(s) \le \max\left\{ s t_0, \sup_{t\ge t_0} ts - Mt^2 \right\}
    \le  s t_0 + \frac{s^2}{4M},
  \]
  so that $\limsup_{s\to\infty} \tilde\Psi^*(s)/s^2\le 1/(4M)$. It follows that
  \[
    \lim_{t\to +\infty} \frac{\tilde\Psi^*(\Phi(t))}{\Phi(t)^2} =
    0 = \lim_{t\to+\infty} \frac{t^2}{\Phi(t)^2} ,
  \]
  showing that $\Phi$ is superlinear. Therefore
  thanks to Dunford-Pettis' theorem, $D\varphi_n$ is precompact in
  $L^1(B;\R^d)$ (and, up to constants, $\varphi_n$ is precompact in
  $W^{1,1}(B)$, not just in $BV(B)$). Consider $\varphi\in W^{1,1}(B)$
and a subsequence $\varphi_{n_k}$ such that $D\varphi_{n_k}\wto D\varphi$
as $k\to+\infty$, weakly in $L^1(B;\R^d)$.  Unfortunately,
 it remains unclear that, without further assumption,
    one can show that $D\varphi_n \mu_n \wtos D\varphi\mu$ as measures.
 
    With the additional assumption that  $1/\mu_n \wto 1/\mu$ in $L^1(B)$,
the situation is much simpler. Indeed, this assumption is
quite restrictive: assuming both $\mu_n\wto\mu$ and $1/\mu_n\wto 1/\mu$ in $L^1$ yields \textit{strong} convergence (this holds in fact
  as soon as $f(\mu_n)\wto f(\mu)$ for a strictly convex $f$ and is classically shown
  using Young measure arguments). In our case, an
  elementary way to see it\footnote{suggested by \texttt{claude.ai}, ``Sonnet 5''.} is
  to write, for any $M>0$:
  \[
  \int_{B\cap \{\mu\le M\}} \frac{(\mu_n-\mu)^2}{\mu_n} \,dx =
  \int_{B\cap \{\mu\le M\}} \mu_n -2\mu  + \frac{\mu^2}{\mu_n}\,dx \stackrel{n\to\infty}{\longrightarrow} 0
  \]
    thanks to our assumptions.
    In particular, 
    \[\int_{B\cap \{\mu\le M\}} |\mu_n-\mu| dx \le \left(\int_B \mu_n \,dx \int_{B\cap \{\mu\le M\}}\frac{|\mu_n-\mu|^2}{\mu_n} \,dx\right)^{1/2} \to 0.\]
    Since by equiintegrability,
    \[ \limsup_{M\to\infty}\sup_{n\ge 1}\int_{B\cap \{\mu> M\}} |\mu_n-\mu| \,dx
    \le 
    \limsup_{M\to\infty}\sup_{n\ge 1}\int_{B\cap \{\mu> M\}} \mu_n+\mu\,dx
    = 0,
    \]
    it follows that $\mu_n\to\mu$ in $L^1(B)$. Then,
    using that for $a,b\ge 0$, $|\sqrt{a}-\sqrt{b}|^2\le |a-b|$,
    \[
    \int_B (\sqrt{\mu_n}-\sqrt{\mu})^2 \,dx \le  \int_B |\mu_n-\mu|\,dx\to 0
    \]
Exchanging the roles of $\mu_n$ and $1/\mu_n$ we get that also $1/\mu_n\to 1/\mu$ in $L^1(B)$, and
$1/\sqrt{\mu_n}\to 1/\sqrt{\mu}$ in $L^2(B)$.

 Using that $v_n\sqrt{\mu_n}$ is bounded in $L^2(B)$, we 
 extract a further subsequence (not relabelled) such that
 $v_{n_k}\sqrt{\mu_{n_k}}$ converges weakly to some limit $w$
  in $L^2(B;\R^d)$.
  Given $\psi\in C^1_c(B;\R^d)$,
  \[
  \lim_{k\to\infty }
  \int_B v_{n_k}\sqrt{\mu_{n_k}} \cdot\frac{\psi}{\sqrt{\mu_{n_k}}}  \,dx
  = \int_B w\cdot \frac{\psi}{\sqrt{\mu}} \,dx
  =  \int_B D\varphi\cdot\psi \,dx.
  \]
The first limit is by strong convergence
  of $\psi/\sqrt{\mu_{n_k}}\to \psi/\sqrt{\mu}$ and the
  second because $v_{n_k}\wto D\varphi$ in $L^1(B;\R^d)$.
  This shows that $w=\sqrt{\mu}D\varphi$.
  On the other hand, using now the strong convergence 
  $\psi\sqrt{\mu_{n_k}}\to \psi\sqrt{\mu}$ in $L^2(B;\R^d)$,
  we obtain:
  \[
\lim_{k\to\infty}  \int_B v_{n_k}\sqrt{\mu_{n_k}} \cdot {\psi}{\sqrt{\mu_{n_k}}}  \,dx = \int_B w \cdot \psi \sqrt{\mu} \,dx
= \int_B \psi\cdot D\varphi\mu \,dx .
  \]
  showing that  $v_{n_k}\mu_{n_k} \wtos D\varphi \mu$ as measures.
  This shows the lemma.
  \end{proof}

\section{Examples for Section~\ref{sec:smirnov}}

In the following, we include two examples corresponding to Section~\ref{sec:smirnov}. The first one in Appendix~\ref{app:non_uniqueness_rect} shows that weak rectified couplings are in general not unique. The second one in Appendix~\ref{app:not_lsc} shows that the flow matching loss $\gamma\mapsto\inf_{w_t\in L^2(\mu_t)}\mathcal L(w_t|\gamma)$ is not lower semicontinuous with respect to weak convergence.

\subsection{Non-uniqueness of Weak Rectified Couplings}\label{app:non_uniqueness_rect}

We have seen in Section~\ref{subsec:weak_rec} that if the flow ODE has a unique solution, then
the weak rectified coupling is unique and coincides with the strong rectified coupling.
However, if the strong rectified coupling does not exist, the weak rectified coupling might be non-unique. We
demonstrate this property using the example from \cite[Section 4.2]{HCD2025}. To this end, let $\mu_0=\mu_1=\mathcal N(0,I)$ (in $\R^d$)
and define $\gamma$ by
$$
\int \psi(x,y)\d\gamma(x,y)=\int \psi(x,-x)\d\mu_0(x).
$$
Then, by \cite[Prop 13]{HCD2025} the velocity field $v_t=\argmin_{w_t\in L^2(\mu_t)}\mathcal L(w_t|\gamma)$ is given by $v_t(x)=-\frac{2}{1-2t}x$.
We define the curve 
$$
\lambda_{x,y}(t)=\begin{cases}(1-2t)x,&\text{for }t\leq\frac12,\\(2t-1)y,&\text{else,}\end{cases}
$$
which connects the points $x$ at time $0$, $0$ at time $\frac12$ and $y$ at time $1$ with straight lines.
Then, we can show that for any coupling $\tilde\gamma\in\Gamma(\mu_0,\mu_1)$ the measure $\Lambda_{\tilde\gamma}$ defined by 
$$
\Lambda_{\tilde\gamma}(A)=\tilde\gamma\left(\left\{(x,y)\in\R^d\times\R^d:\lambda_{x,y}\in A \right\}\right)
$$
fulfills \eqref{eq:Lambda1} and \eqref{eq:Lambda2}.
Since
$$
\int\psi(x,y)\d\tilde\gamma(x,y)=\int \psi(\lambda(0),\lambda(1))\d\Lambda_{\tilde \gamma}(\lambda)
$$
we get that $\tilde\gamma\in\mathcal R(\gamma)$. In particular, in this example, we have that $\mathcal R(\gamma)=\Gamma(\mu_0,\mu_1)$.
\smallskip

We note that this example also provides a limit point of reflow, which is straight, but not rectifiable as it fulfills $\gamma\in\mathcal R(\gamma)$. As pointed out in \cite[Prop 13]{HCD2025}, the velocity field is also a gradient, so that all these conclusions also hold for $\mathcal R_p$ instead of $\mathcal R$.

\subsection{Non-lower-semicontinuity of the Flow Matching Loss}\label{app:not_lsc}

We construct a sequence $(\gamma_n)_n$ of couplings $\gamma_n\in\mathcal P_2(\R^d\times\R^d)$ converging (weakly) to some
$\gamma$ such that
$$
\min_{w_t\in L^2(\mu_{t,n})} \mathcal L(w_t|\gamma_n)=0,\quad\text{and}\quad\min_{w_t\in L^2(\mu_{t})} \mathcal L(w_t|\gamma)>0.
$$
This implies in particular that $\gamma\mapsto \min_{w_t\in L^2(\mu_{t})} \mathcal L(w_t|\gamma)$ is not lsc with respect to weak convergence.
We define $\gamma_n$ and $\gamma$ on $\R^d\times \R^d$ for $d=2$ by
$$
\gamma_n=\tfrac{1}{\sqrt2}\Haus^1\mres \bigcup_{x\in[\frac12,1]} (0,x,\tfrac{1}{n}, x-1) + \tfrac{1}{\sqrt2}\Haus^1\mres \bigcup_{x\in[\frac12,1]} (0,x-1, -\tfrac{1}{n}, x)
$$
and 
$$
\gamma=\tfrac{1}{\sqrt2}\Haus^1\mres \bigcup_{x\in[\frac12,1]} (0,x,0, x-1) + \tfrac{1}{\sqrt2}\Haus^1\mres \bigcup_{x\in[\frac12,1]} (0,x-1, 0, x).
$$
It is straightforward to see that $\gamma_n\rightharpoonup\gamma$ (due to uniformly bounded moments, this is even true in Wasserstein).
From this definition, we derive that the intermediate measures $\mu_{t,n}$ are given by
$$
\mu_{t,n}=\Haus^1\mres \bigcup_{x\in[\frac12,1]}(\tfrac{t}{n},x-t)+\Haus^1\mres \bigcup_{x\in[\frac12,1]}(-\tfrac{t}{n},x-1+t)
$$
and velocity fields minimizing $\mathcal L(\cdot|\gamma_n)$ given by
$$
v_{t,n}(x)=\begin{cases}
(\tfrac1n,-1),&\text{for $x=(\tfrac{t}{n},z)$ if $z+t\in[\tfrac12,1]$},\\
(-\tfrac1n,1),&\text{for $x=(-\tfrac{t}{n},z)$ for some $z+1-t\in[\tfrac12,1]$}.
\end{cases}
$$
(This form of $v_t$ can be derived by the representation of the velocity as conditional expectation, see, e.g., \cite{LCL2023}.)
One easily finds that for this velocity we get $\mathcal L(v_{t,n}|\gamma_n)=0$.
For the limit plan $\gamma$ we obtain (also by the representation of the velocity as conditional expectation) that 
$$
v_{t}(x)=\begin{cases}
(0,-1),&\text{for $x=(0,z)$ if $z+t\in[\tfrac12,1]$ and $z+1-t\not\in[\tfrac12,1]$},\\
(0,1),&\text{for $x=(0,z)$ if $z+1-t\in[\tfrac12,1]$ and $z+t\not\in[\tfrac12,1]$},\\
0,&\text{else.}
\end{cases}
$$
Inserting this form into the loss function yields $\mathcal L(v_t|\gamma)>0$.
We observe that all velocity fields have a potential (which can be extended to $\R^d$), so that the same result holds for the rectified coupling with gradient constraint.

\section{Proof of Proposition~\ref{prop:limits_coincide}}\label{app:limits_coincide}

Let $(\gamma_n)_n$ and $(\gamma_n^{(N)})_n$ be generated by Algorithm~\ref{alg:minibatch_OT_rec} and let $(n_k)_k$ be a subsequence such that $\gamma_{n_k}\wto \gamma$. We show that any limit point of $(\gamma_{n_k}^{(N)})_k$ and $(\gamma_{n_k+1})_k$ coincides with $\gamma$.
Then, this implies $\gamma_{n_k}^{(N)}\wto \gamma$ and $\gamma_{n_k+1}\wto \gamma$ and by induction this yields also $\gamma_{n_k+l}\wto\gamma$ and $\gamma_{n_k+l}^{(N)}\wto\gamma$ for any $l\in\Z_{\geq0}$.

Due to weak compactness of $\Gamma(\mu_0,\mu_1)$ we can go over to another subsequence of $(n_k)_k$ (which we again denote by $(n_k)_k$) such that $\gamma_{n_{k}}^{(N)}\wto \gamma^{(N)}$ and $\gamma_{n_{k}+1}\wto\gamma'$ for some $\gamma^{(N)}$ (not necessarily $\mathcal{F}_N(\gamma)$) and $\gamma'$. We have to show that this implies $\gamma^{(N)}=\gamma'=\gamma$.

\subsection{Proof of $\gamma^{(N)}=\gamma'$}

We prove the more general statement formulated in the following lemma, which we prove in the rest of this subsection. Applying it to the sequences $\tilde \eta_k=\gamma^{(N)}_{n_k}$ and $\eta_k=\gamma_{n_k+1}$ yields $\gamma^{(N)}=\gamma'$.

\begin{lemma}\label{lem:gamma_N_equals_gamma_prime}
Assume that $\tilde \eta^n\in\Gamma(\mu_0,\mu_1)$ and $\eta^n\in\mathcal R(\tilde\eta^n)$ fulfill that $\int\|x-y\|^2\d \tilde\eta^n(x,y)\geq\int\|x-y\|^2\d \eta^n(x,y)\geq \int\|x-y\|^2\d \tilde\eta^{n+1}(x,y)$ for all $n$. Assume also that $\tilde\eta^n$ converges weakly to some $\tilde\eta\in\Gamma(\mu_0,\mu_1)$. 
Then, $\eta^n$ converges to $\tilde\eta$. Moreover, the minimizers $\tilde v_t^n$ and $\tilde v_t$ of the loss functions $\mathcal L(\cdot|\tilde\eta^n)$ and $\mathcal L(\cdot|\tilde\eta)$ fulfill $\tilde v_t^n\tilde\mu_t^n \otimes dt\wto \tilde v_t\tilde \mu_t \otimes dt$.
\end{lemma}
\begin{proof}
We denote by $\tilde\mu_t^n$ the interpolations \eqref{eq:def_mu_t} of $\tilde\eta^n$.
Then, the rectification $\eta^n$ of $\tilde\eta^n$ in Section~\ref{sec:smirnov} is built upon the sequence
of measures $\Lambda^n$ on $C^0([0,1];\R^d)$ given 
by~\cite[Thm~8.2.1]{AGS2008}  which satisfy,
for all $t$ and any test function $\psi$:
\begin{equation}\label{eq:Lambdat}
\int \psi(\lambda(t))\d\Lambda^n (\lambda)= \int\psi\d\tilde \mu^n_t
\end{equation}
and
\[
\dot{\lambda}(t) = \tilde v^n_t(\lambda(t))
\]
for a.e.~$t$ and $\Lambda^n$-a.e.~$\lambda$, where $\tilde v^n_t$ satisfies
the continuity equation:
\[ \partial_t \tilde \mu_t^n + \Div \tilde v^n_t \tilde \mu^n_t = 0.\]

We have that the Wasserstein cost is non-increasing, and in particular
the assumptions of the Lemma yield that:
\begin{equation}\label{eq:Wassequal}
\lim_n \int \|x-y\|^2\d \tilde\eta^n(x,y) - \int \|x-y\|^2 \d{\eta}^{n}(x,y) = 0.
\end{equation}
The coupling ${\eta}^{n}$ is defined from $\Lambda^n$ by
\[
\int \psi(x,y)\d{\eta}^{n}(x,y)= \int \psi(\lambda(0),\lambda(1))\d\Lambda^n(\lambda) 
\]
for any test function $\psi$.

Recall that by the first-order optimality condition of $v_t$ in the loss \eqref{eq:flow_matching_loss}, one has:
\[
\int \langle y-x, \tilde v^n_t((1-t)x+ty)\rangle \d \tilde\eta^n(x,y) =
\int \|\tilde v_t^n((1-t)x+ty)\|^2 \d \tilde\eta^n(x,y).
\]
Also, $\int_0^1\dot\lambda(t)\d t=\lambda(1)-\lambda(0)$ so that $\int_0^1\langle \dot\lambda(t),\lambda(1)-\lambda(0)\rangle \d t=\|\lambda(1)-\lambda(0)\|^2$. Together with \eqref{eq:Lambdat} this implies
{\small\begin{equation}\label{eq:energycontrol}
\begin{aligned}
&\int\|x-y\|^2\d\tilde\eta^n(x,y)-\int\|x-y\|^2\d\eta^n(x,y)\\
&=\int_0^1\int\|x-y\|^2-\|\tilde v_t^n((1-t)x+ty)\|^2\d\tilde\eta^n\d t\\
&\ \hspace{4.5cm} +\int_0^1\int\|\tilde v_t^n(x_t)\|^2 \d\tilde\mu_t^n \d t -
\int \|\lambda(1)-\lambda(0)\|^2\d\Lambda^n(\lambda)\\
&=\int_0^1\int\|y-x-\tilde v_t^n((1-t)x+ty)\|^2\d\tilde\eta^n\d t +\int_0^1\int \|\dot\lambda(t)\|^2-\|\lambda(1)-\lambda(0)\|^2\d\Lambda^n(\lambda)\d t\\
&=\int_0^1\int\|y-x-\tilde v_t^n((1-t)x+ty)\|^2\d\tilde\eta^n\d t+\int_0^1\int \|\dot\lambda(t)-(\lambda(1)-\lambda(0))\|^2\d\Lambda^n(\lambda)\d t.
\end{aligned}
\end{equation}
}
\begin{remark}
In Section~\ref{sec:gradient}, the velocity field minimizing \eqref{eq:flow_matching_loss} is replaced by a projected version $v_t^{p,n}$. Then, we have $\int \|\tilde v_t^n\|^2\d\tilde \mu^n_t=\int \|\tilde v_t^n-\tilde v_t^{p,n}\|^2\d\tilde \mu^n_t + \int \|\tilde v_t^{p,n}\|^2\d\tilde\mu^n_t$ so that the same holds with an additional positive term on the right-hand side.
The first term in the right-hand side of~\eqref{eq:energycontrol} controls how far the coupling $\tilde\eta^n$ is from a deterministic coupling, while the second controls how straight the new coupling is.
\end{remark}
Note that for any rectifiable path $\lambda$,
\begin{multline*}
\|\lambda(t) - ((1-t)\lambda(0) + t\lambda(1))\|^2 = 
\| \lambda(t)-\lambda(0) - t(\lambda(1)-\lambda(0))\|^2
\\= \left\|\int_0^t\big(\dot\lambda(s) - (\lambda(1)-\lambda(0))\big)\d s\right\|^2
\le t\int_0^t \|\dot\lambda(s)- (\lambda(1)-\lambda(0))\|^2\d s
\end{multline*}
so that
\begin{multline*}
\int_0^1 \|\lambda(t) - ((1-t)\lambda(0) + t\lambda(1))\|^2 \d t
\le \int_0^1  \tfrac{1-s^2}{2}\|\dot\lambda(s)- (\lambda(1)-\lambda(0))\|^2\d s
\\ \le \frac{1}{2}
\int_0^1  \|\dot\lambda(s)- (\lambda(1)-\lambda(0))\|^2\d s.
\end{multline*}
We obtain that in addition to~\eqref{eq:energycontrol}, one also
has
\begin{equation}\label{eq:controlstraightness}
    \int \|x-y\|^2 \d\tilde\eta^n  - \int \|x-y\|^2 \d\eta^n
\ge  2
 \int_0^1 \int \|\lambda(t) - ((1-t)\lambda(0) + t\lambda(1))\|^2\d\Lambda^n(\lambda) \d t,    
\end{equation}
where the left side goes to zero by \eqref{eq:Wassequal}.
\begin{remark} Since $\psi\mapsto \int \psi(\lambda(t),(1-t)\lambda(0)+t\lambda(1)) \d\Lambda^n$ is a coupling between the interpolations $\tilde \mu_t^n$ and $\mu_t^{n}$ from \eqref{eq:def_mu_t} wrt $\tilde\eta^n$ and $\eta^n$, we already find that
\[
2\int_0^1 W_2^2(\tilde\mu_t^n,\mu_t^{n}) \d t\leq \int \|x-y\|^2 \d\tilde\eta^n  - \int \|x-y\|^2 \d\eta^{n}
\to 0.
\]
However we will prove the stronger result that the limit points of $\tilde\eta^n$ and
$\eta^n$ must coincide.
\end{remark}

Since $\tilde\eta^n\wto\tilde\eta$, we know that also $\tilde \mu_t^n$ converges to the $\tilde\mu_t$ corresponding to $\tilde\eta$ for all $t$,
and $\tilde v^{n}_t\tilde\mu^{n}_t\wto \tilde v_t\tilde\mu_t$ which satisfies
the continuity equation, cf.~\cite[Thm~5.4.4]{AGS2008}.

Now, assume $\eta^{n}$ also converges to a weak limit $\eta$ (which is always true up to subsequences).
Let $\psi(x,y)$ be a smooth test function with compact support. Then,
\begin{equation}\label{eq:limgamma}
\begin{aligned}
\int &\psi(x,y)  \d \tilde\eta^n(x,y)=   \int_0^1 \int \psi(x,y)\d \tilde\eta^n(x,y) \d t\\ 
=& \int_0^1\int \psi(x_t-t\tilde v_t^n(x_t),x_t+(1-t)\tilde v_t^{n}(x_t)) \d \tilde\mu_t^{n}(x_t)\d t
\\ &+ \int_0^1\int \big(\psi(x,y)-\psi(x_t-t \tilde v_t^n(x_t),x_t+(1-t)\tilde v_t^n(x_t))\big)\d\tilde\eta^n(x,y)\d t 
\end{aligned}
\end{equation}
(where in the last line, ``$x_t$'' is not a variable but
a shorthand for $(1-t)x+ty$).
Since
\begin{align*}
    \psi(x,y)-\psi(x_t-t \tilde v_t^n(x_t),x_t+ &(1-t) \tilde  v_t^n  (x_t))
    \le \|D\psi\|_\infty \left\| 
    \begin{pmatrix}
    x_t-t \tilde v_t^n(x_t) -x\\ x_t+(1-t)\tilde v_t^n(x_t) - y
    \end{pmatrix}
    \right\| \\
    & \le  \|D\psi\|_\infty \sqrt{t^2+(1-t)^2}\|\tilde v_t^n(x_t)-(y-x)\|,
\end{align*}
the last error term goes to zero thanks to~\eqref{eq:energycontrol}.
Then,
\begin{multline*}
\int_0^1\int \psi(x_t- t\tilde v_t^n(x_t),x_t+(1-t)\tilde v_t^{n}(x_t)) \d \tilde \mu_t^{n}(x_t)\d t
\\= \int_0^1 \int\psi(\lambda(t)- t \dot\lambda(t),\lambda(t)+(1-t)\dot\lambda(t))\d\Lambda^n(\lambda)\d t
\end{multline*}
and we use now that
\begin{multline*}
    |\psi(\lambda(t)-t\dot\lambda(t),\lambda(t)+(1-t)\dot\lambda(t))
     - \psi(\lambda(0),\lambda(1))|
     \\\le \|D\psi\|_\infty  \left( 2\|\lambda(t)-((1-t)\lambda(0)+t\lambda(1))\| +\sqrt{t^2+(1-t)^2} \|\dot\lambda(t)-(\lambda(1)-\lambda(0))\|\right)
\end{multline*}
to show, using again~\eqref{eq:energycontrol} and~\eqref{eq:controlstraightness} that
\begin{align*}
\int_0^1\int \psi(x_t-t\tilde v_t^n(x_t),x_t+(1-t)\tilde v_t^{n}(x_t)) &\d \tilde \mu_t^{n}(x_t)\d t
\\ &= \int_0^1 \int \psi(\lambda(0),\lambda(1))\d\Lambda^n(\lambda)  \d t+ o(1)
\\ &=\int \psi(x,y) \d\eta^n(x,y)+ o(1).
\end{align*}
We find in the limit (still, a priori, along a subsequence)
in~\eqref{eq:limgamma} that:
\[
\int \psi(x,y)\d\tilde\eta(x,y)
=\int \psi(x,y)\d\eta (x,y)
\]
which proves that $\eta^n\wto\tilde\eta$.
One also deduces from~\eqref{eq:energycontrol}
(precisely, from the fact that 
$\int_0^1\int\|y-x-\tilde v_t^n((1-t)x+ty)\|^2\d\tilde\eta^n\d t $
goes to zero) that:
\[
\int_0^1\int \psi(x_t,t) \tilde v_t(x_t)\d\tilde \mu_t(x_t)\d t
= \int_0^1\int \psi((1-t)x+ty,t) (y-x) \d\tilde \eta(x,y)\d t,
\]
for any test function $\psi$, which shows that $\tilde v_t=\argmin\mathcal{L}(\cdot|\tilde\eta)$.
\end{proof}

\subsection{Proof of $\gamma^{(N)}=\gamma$}

Let $\Sigma\subset \R^{N\times N}$ be the convex set of doubly stochastic matrices:
\[
  \Sigma := \left\{ X\in \R^{N\times N}: X\ge 0, X\ones = \ones, X^T\ones=\ones\right\}
\]
where ``$X\ge 0$'' is intended componentwise ($X_{i,j}\ge 0\,\forall i,j$) and
$\ones=(1,\dots,1)^T\in\R^N$ is the $N$-vector with all components equal to $1$. It
is well known that equivalently, $\Sigma$ can be defined as the convex envelope
of the permutation matrices $(\delta_{\sigma(i),j})_{i,j=1}^N$, $\sigma\in\Sigma_N$.

The next lemma shows that $\gamma^{(N)}$ is ``some version of minibatch OT'' applied to $\gamma$. However, since the discrete OT mapping within minibatch OT is not unique, we cannot conclude in general that $\gamma^{(N)}=\mathcal F_N(\gamma)$.

\begin{lemma}\label{lem:X-ists}
Let $(\eta_n)_n$ be a sequence in $\Gamma(\mu_0,\mu_1)$, denote by $\eta_n^{(N)}\coloneqq\mathcal F_N(\eta_n)$ and assume that $\eta_n\wto\eta$ and $\eta_n^{(N)}\wto \eta^{(N)}$ for some $\eta$ and $\eta^{(N)}$. Then, there exists a Borel map $X\colon(\R^d\times\R^d)^N\to\Sigma$ such that 
\[
\int\psi(x,y)\d\eta^{(N)}(x,y)=\frac1N\int \sum_{i,j=1}^N \psi(x^{(i)},y^{(j)})X_{i,j}(\Bx,\By)\d\eta^{\otimes N}(\Bx,\By).
\]
\end{lemma}
\begin{proof}
We define for $\eta_n$ the Borel map $X_n\colon(\R^d\times\R^d)^N\to\Sigma$ by $(X_n)_{i,j}(\Bx,\By)=1$ if $j=\sigma(i)$ for the permutation $\sigma$ used in the discrete OT operator $S$ within the minibatch OT. Then, we have by definition that
\begin{equation}\label{eq:minibatch_OT_in_X_notation}
\int\psi(x,y)\d\eta_n^{(N)}(x,y)=\frac{1}{N}\int \sum_{i,j=1}^N\psi(x^{(i)},y^{(j)})(X_n)_{i,j}(\Bx,\By)\d\eta_n^{\otimes N}(\Bx,\By).
\end{equation}
Now let $\lambda$ be a limit point of $\lambda_n\coloneqq X_n\eta_n^{\otimes N}$ (which exists due to compactness arguments).
Then, one has for any non-negative continuous test function $\psi$ that
 \[
   \left\|\int \psi(\Bx,\By)\d \lambda_n(\Bx,\By)\right\|
   \le \int \psi(\Bx,\By)  \d\eta^{\otimes N}_n(\Bx,\By)
   \to \int\psi(\Bx,\By) \d\eta^{\otimes N}(\Bx,\By)   
 \]
 from which we deduce that $|\lambda| \le \eta^{\otimes N}$ and in particular
 $\lambda \pp \eta^{\otimes N}$
 as Radon measures. Hence, one has $\lambda = X\eta^{\otimes N}$ for
 some $X\in L^1_{\text{loc}}(\eta^{\otimes N}; \R^{N\times N})$. In addition
 (by convexity, or integrating against functions of the form $\varphi(x,y)\ones$)
 one checks easily that $X(\Bx,\By)\in \Sigma$ almost everywhere.
 
Finally, we obtain from \eqref{eq:minibatch_OT_in_X_notation} that for the subsequence $n_k$ such that $X_{n_k}\eta_{n_k}^{\otimes N}\wto X\eta^{\otimes N}$ it holds
\begin{align*}
\int \psi(x,y)\d\eta^{(N)}(x,y)&=\lim_{k\to\infty}\int \psi(x,y)\d\eta_{n_k}^{(N)}(x,y)\\
&=\lim_{k\to\infty}\frac1N\int\sum_{i,j=1}^N \psi(x^{(i)},y^{(j)})(X_{n_k})_{i,j}(\Bx,\By)\d\eta_{n_k}^{\otimes N}(\Bx,\By)\\
&=\frac1N\int\sum_{i,j=1}^N \psi(x^{(i)},y^{(j)})X_{i,j}(\Bx,\By)\d \eta^{\otimes N}(\Bx,\By).
\end{align*}
\end{proof}
\begin{remark}
    One easily shows that $X$ is minimizing the Wasserstein cost, that is,
    \[
    \frac1N\int\sum_{i,j=1}^N \|x^{(i)}-y^{(j)}\|^2 X_{i,j}(\Bx,\By)\d \eta^{\otimes N}(\Bx,\By)
    = \int \|x-y\|^2 d\eta^{(N)},  
    \]
    so that in particular $X(\Bx,\By)=\mathrm{Id}$ at $\eta^{\otimes N}$-almost all points of $(\R^d\times\R^d)^N$ where
    the minimizer of the cost is unique.
\end{remark}
Since we have by Lemma~\ref{lem:gamma_N_equals_gamma_prime} that $\gamma^{(N)}=\gamma'$ and by Corollary~\ref{cor:limit_points_minibatch_reflow} that $\gamma'$ is monotone, we obtain that $\gamma^{(N)}$ is supported on a graph of a Borel map.
Under this additional condition, we finally obtain that $\gamma^{(N)}=\gamma$.
We formalize this observation in the following lemma. Choosing the sequence $(\eta_n)_n$ as $(\gamma_{n_{k}})_k$ yields the claim.

\begin{lemma}
    Let $(\eta_n)_n$ be a sequence in $\Gamma(\mu_0,\mu_1)$, denote by $\eta_n^{(N)}\coloneqq\mathcal F_N(\eta_n)$ and assume that $\eta_n\wto\eta$ and $\eta_n^{(N)}\wto \eta^{(N)}$ for some $\eta$ and $\eta^{(N)}$. Further assume that there exists a (Borel) transport map $u(x)$ such
    that $\eta^{(N)} = (I,u)_\sharp \mu_0$ and that $\mu_1$ 
    has no atoms. Then we have $\eta=\eta^{(N)}$ (and $u$ is $N$-monotone).
\end{lemma}
\begin{proof}
By Lemma~\ref{lem:X-ists} there exists a $\Sigma$-valued map $X(\Bx,\By)$ such that for any
    continuous $\psi$ with compact support,
    \begin{align*}
\int_{\R^d}\psi(x,u(x))\d\mu_0(x) &= \int_{\R^d\times\R^d} \psi(x,y)\d\eta^{(N)}(x,y)\\
&= \frac{1}{N} \int_{(\R^d\times \R^d)^N} 
 \left\langle (\psi(x^{(i)},y^{(j)}))_{i,j},
     X(\Bx,\By)\right\rangle \d\eta^{\otimes N}(\Bx,\By).
    \end{align*}

     Letting $G= \{(x,u(x)): x\in\R^d\}$, which is Borel as the graph of a Borel map (see~\cite[Prop~12.4]{Kechris1995}), we see that $\eta^{(N)} \mres (\R^d\times\R^d)\setminus G = 0$.
Hence, if $E=(\R^d\times\R^d)\setminus G$, (using the regularity of the measures)
\[
0 =  \int_{(\R^d\times \R^d)^N} 
 \sum_{i,j}\chi_E(x^{(i)},y^{(j)})
     X_{i,j}(\Bx,\By) \d\eta^{\otimes N}(\Bx,\By)
\]
 which shows in particular that for $\eta^{\otimes N}$-a.e.~$(\Bx,\By)$, if $X_{i,j}(\Bx,\By)>0$ then $(x^{(i)},y^{(j)})\in G$, that is, $y^{(j)}=u(x^{(i)})$. 
 If $E_{i,j} = \{ (\Bx,\By) \in (\R^d\times \R^d)^N: X_{i,j}(\Bx,\By)>0\}$
 and $i\neq j$,
 one sees then that
 \begin{multline*}
 \eta^{\otimes N}(E_{i,j}) \le \int_{(\R^d\times\R^d)^2} \chi_{\{ y^{(j)} = u(x^{(i)})\}}
 \d\eta(x^{(i)},y^{(i)})\d\eta(x^{(j)},y^{(j)})
  \\= \int_{\R^d}\left(\int_{\{u(x^{(i)})\}} \d\mu_1(y^{(j)}) \right)\d\mu_0(x^{(i)})=0,
 \end{multline*}
 since we assumed $\mu_1$ has no atoms. We deduce
 that for $\eta^{\otimes N}$-a.e.~$(\Bx,\By)$,
  $X(\Bx,\By)=\mathrm{Id}$ so that $\eta^{(N)}=\eta$. As a consequence, $u$ is $N$-monotone.
\end{proof}

\section{Non-Optimality of Minibatch-OT Reflow for Disconnected Supports}\label{app:wrong_with_disconnected_support}

The following lemma shows that the example from Section~\ref{sec:numerics} is $N$-cyclically monotone.

\begin{lemma}\label{lem:drehdich_ist_monoton}
    Let $M=2N+1$ for some $N\in\Z_{>1}$ and let $\gamma$ be defined as above for $\epsilon$ small enough. Then $\gamma$ is $N$-cyclically monotone.
\end{lemma}
\begin{proof}
We consider the points $x_k$ from \eqref{eq:drehdich}. Further, we extend the set of indices to $\Z$ by $x_{k+lM}=x_k$ for all $l\in \Z$.
We also use the periodic distance of indices defined as $d(k,l)=\min_{j\in \Z}|k-l+jM|$. Since $M=2N+1$ is odd, every difference $k-l$ of indices admits a unique representative $s\equiv k-l\pmod{M}$ with $s\in\{-N,...,N\}$, and it holds $d(k,l)=|s|$.

Denote by $\delta=\|x_0-x_1\|$ (with $x_k$ from \eqref{eq:drehdich}).
We first estimate $\|x_0-x_k\|^2$ for $k\in \{0, 1,...,N\}$. For $k=0$ and $k=1$ we directly obtain
that $\|x_0-x_k\|=k\delta$.
For the general case, we have
\begin{align*}
\|x_0-x_k\|^2&=\Big(1-\cos\Big(\frac{2\pi k}{M}\Big)\Big)^2+\sin\Big(\frac{2\pi k}{M}\Big)^2=4\sin\Big(\frac{\pi k}{M}\Big)^2.
\end{align*}
Using the inequality $\sin(\alpha\frac{\pi}{2})\geq \alpha$ for $\alpha\in[0,1]$ we obtain for $k\geq 3$
$$
\|x_0-x_k\|^2\geq\frac{16k^2}{M^2}\geq k\frac{4\pi^2}{M^2} > k\|x_0-x_1\|^2=k\delta^2,
$$
where the last inequality estimates the distance of $x_0$ and $x_1$ by the length of the circle arc between them.
For $k=2$, we observe that
$$
\|x_0-x_2\|^2-2\delta^2=4\sin(\frac{2\pi}{M})^2-8\sin(\frac{\pi}{M})^2=8\sin(\frac{\pi}{M})^2(2\cos(\frac{\pi}{M})^2-1).
$$
Since $M\geq 5$, we have that all terms are strictly positive and $\|x_0-x_2\|^2>2\delta^2$.
Denoting $\rho=\min\Big(\frac12\|x_0-x_2\|^2,\frac{4\pi^2}{M^2}\Big) - \|x_0-x_1\|^2>0$, we obtain
 that for $k\geq 2$, we have that $\|x_0-x_k\|^2 > k\delta^2+\rho$.
We now assume that $\epsilon$ is small enough such that $8N\epsilon<\rho$.

Now, we assume that there exist $(x^{(i)},y^{(i)})\in\mathrm{supp}(\gamma)$ for $i=1,...,N$ such that
\begin{equation}\label{eq:assumption_not_cyc_mon}
\sum_{i=1}^N \|x^{(i)}-y^{(i+1)}\|^2< \sum_{i=1}^N \|x^{(i)}-y^{(i)}\|^2.
\end{equation}
Then, it holds $y^{(i)}=u(x^{(i)})$ and using the notation $k(i)$ such that $\|x_{k(i)}-x^{(i)}\|\leq \epsilon$, we obtain that $\|x_{k(i)+1}-y^{(i)}\|\leq \epsilon$. We further use the notation $k(N+1)=k(1)$ and denote by $\sigma_i\in\{-N,...,N\}$ the unique representative of $k(i+1)+1-k(i)$ modulo $M$, so that $d(k(i),k(i+1)+1)=|\sigma_i|$. Since $\sum_{i=1}^N \big(k(i+1)+1-k(i)\big)=N$, this implies in particular that $\sum_{i=1}^N \sigma_i\equiv N \pmod{2N+1}$; note that every integer congruent to $N$ modulo $2N+1$ has absolute value at least $N$.
Moreover, since $\|x_{k(i)}-x_{k(i+1)+1}\|\leq\|x^{(i)}-y^{(i+1)}\|+2\epsilon$ and $\|x_{k(i)}-x_{k(i+1)+1}\|\leq 2$, we have
\begin{equation}\label{eq:center_perturbation}
\|x_{k(i)}-x_{k(i+1)+1}\|^2\leq \|x^{(i)}-y^{(i+1)}\|^2+4\epsilon\,\|x_{k(i)}-x_{k(i+1)+1}\|\leq \|x^{(i)}-y^{(i+1)}\|^2+8\epsilon,
\end{equation}
where the first inequality is trivial for $\|x_{k(i)}-x_{k(i+1)+1}\|\leq 2\epsilon$ and follows otherwise from $\|x^{(i)}-y^{(i+1)}\|^2\geq\big(\|x_{k(i)}-x_{k(i+1)+1}\|-2\epsilon\big)^2$.

\textbf{Step 1:} We first show that $d(k(i),k(i+1)+1)\leq 1$ for all $i=1,...,N$ (with the convention that $k(N+1)=k(1)$).
To this end, we assume by contradiction that there exists some $i$ such that $d(k(i),k(i+1)+1)\geq 2$. Then, we have
\begin{align*}
    \delta^2 \sum_{i=1}^N d(k(i),k(i+1)+1)
    &\leq \sum_{i=1}^N \|x_{0}-x_{d(k(i),k(i+1)+1)}\|^2 - \rho
    = \sum_{i=1}^N \|x_{k(i)}-x_{k(i+1)+1}\|^2 - \rho\\
    &\leq \sum_{i=1}^N \|x^{(i)}-y^{(i+1)}\|^2 +8N\epsilon - \rho\\
    &< \sum_{i=1}^N \|x^{(i)}-y^{(i)}\|^2 = N \delta^2,
\end{align*}
where the first inequality uses $\|x_0-x_m\|^2\geq m\delta^2$ for $m\in\{0,1\}$ together with $\|x_0-x_m\|^2> m\delta^2+\rho$ for $m\geq 2$, the second inequality uses \eqref{eq:center_perturbation}, and the last step uses $8N\epsilon<\rho$ and \eqref{eq:assumption_not_cyc_mon}.
Thus, we get that $\sum_{i=1}^N d(k(i),k(i+1)+1) < N$ and hence
$$
\Big|\sum_{i=1}^N \sigma_i\Big|\leq \sum_{i=1}^N |\sigma_i| = \sum_{i=1}^N d(k(i),k(i+1)+1)< N,
$$
which contradicts $\sum_{i=1}^N \sigma_i\equiv N \pmod{2N+1}$.

\textbf{Step 2:} Since we have $d(k(i),k(i+1)+1)=|\sigma_i|\leq 1$, we know that $\sum_{i=1}^N \sigma_i\in\{-N,...,N\}$. On the other side, we again know that $\sum_{i=1}^N \sigma_i\equiv N \pmod{2N+1}$ such that we obtain that $\sum_{i=1}^N \sigma_i=N$ and hence $\sigma_i=1$ for all $i$. By the definition of $\sigma_i$, this implies $k(i+1)\equiv k(i)\pmod M$ for all $i$. In particular, we know that there exists some $k=k(1)=\cdots=k(N)$ such that $x^{(i)}\in B_\epsilon(x_k)$. Since on $B_\epsilon(x_k)$, the map $u$ is only a translation (which is cyclically monotone), we obtain that
$$
\sum_{i=1}^N \|x^{(i)}-y^{(i+1)}\|^2 = \sum_{i=1}^N \|x^{(i)}-u(x^{(i+1)})\|^2\geq\sum_{i=1}^N \|x^{(i)}-u(x^{(i)})\|^2= \sum_{i=1}^N \|x^{(i)}-y^{(i)}\|^2,
$$
which contradicts \eqref{eq:assumption_not_cyc_mon} and we are done.
\end{proof}

\subsection*{Acknowledgments}

This research is supported  by the ANR (Agence Nationale de la Recherche), grant ANR-23-PEIA-0004 (PEPR IA, PDE-AI), and ANR-23-IACL-0008 (PR[AI]RIE-PSAI), funded by ``France 2030'' and by the German Research Foundation (DFG, Deutsche
Forschungsgemeinschaft) with project no 530824055.

\bibliographystyle{abbrv}
\bibliography{ref}

\end{document}